\documentclass[aos,preprint]{imsart}

\RequirePackage[OT1]{fontenc}
\RequirePackage{amsthm,amsmath}
\RequirePackage[numbers]{natbib}
\RequirePackage[colorlinks,citecolor=blue,urlcolor=blue]{hyperref}

\startlocaldefs
\numberwithin{equation}{section}
\theoremstyle{plain}
\newtheorem{thm}{Theorem}[section]
\newtheorem{rem}{Remark}
\newtheorem*{lem*}{Lemma}
 \usepackage{amssymb}
\usepackage[table]{xcolor}
\usepackage{mathrsfs}
\usepackage{stmaryrd}
\usepackage{wasysym}
\usepackage{esint}
\endlocaldefs

\begin{document}

\begin{frontmatter}

\title{A Goodness-of-Fit test for  Elliptical Distributions
  with Diagnostic Capabilities }
\runtitle{GoF test for Multivariate Elliptical Distributions}
\thankstext{T1}{The authors would like to thank Bernard Boulerice for his contribution to a preliminary version of this paper.}

\begin{aug}
\author{\fnms{Gilles R.} \snm{Ducharme}\thanksref{m1}\ead[label=e1]{gilles.ducharme@umontpellier.fr}},
\and
\author{\fnms{Pierre} \snm{Lafaye de Micheaux}\thanksref{m2}\ead[label=e2]{lafaye@unsw.edu.au}}

\runauthor{G. Ducharme et al.}

\affiliation{Universit\'e de Montpellier\thanksmark{m1} and UNSW Sydney\thanksmark{m2}}

\address{IMAG, Univ. Montpellier, CNRS, Montpellier, France\\
School of Mathematics and Statistics, UNSW Sydney, NSW 2052 Australia\\
\printead{e1}\\
\phantom{E-mail:\ }\printead*{e2}}

\end{aug}

\begin{abstract}
This paper develops a smooth test of goodness-of-fit for
elliptical distributions. The test is adaptively omnibus,  invariant to affine-linear transformations 
and has a convenient expression that can be broken into components. These components
have diagnostic capabilities and can be used to identify specific
departures. This helps in correcting the null model when the test rejects. As an example,
the results are applied to the multivariate normal distribution for which the R package ECGofTestDx is available.
 It is shown that the proposed
test strategy encompasses and generalizes a number of existing approaches. Some other cases are studied, such as the bivariate  
Laplace, logistic and Pearson type II distribution. 
A simulation experiment shows the usefulness of the diagnostic tools.
\end{abstract}



\end{frontmatter}

\section{Introduction}\label{sec:Introduction}

Elliptically contoured (EC), or elliptical for short, distributions
have become important tools in the analysis of multivariate data.
They retain from the multivariate normal (MVN) distribution the feature
of elliptical symmetry about a location $\boldsymbol{\mu}$. They
extend the MVN to allow modelling data with short or large tails.
Tools of multivariate analysis such as regression, correlation, PCA, discrimination,  are easily
ported to them. Many inferential procedures
from MVN theory remain valid, after slight modifications. They pave the way toward more flexible models, such
as skew-elliptic distributions, or more specialized, such as elliptical
copulas, elliptical complex and elliptical matrix distributions.
Consequently, they are used in many applications, such has portfolio
theory, risk management, radioimmunoassay,
astronomy, physics, signal processing, etc. (see \cite{Chmielewski1981} 
for a bibliography). 

In this paper we focus on $m$-dimensional  EC distributions
with density of the form 

\begin{equation}
f(\boldsymbol{x};\eta)=c_{m}\,\sqrt{\det(\mathbf{V}^{-1})}\,\phi_{m}((\boldsymbol{x}-\boldsymbol{\mu})^{T}\mathbf{V}^{-1}(\boldsymbol{x}-\boldsymbol{\mu})),\label{eq:1.1elliptical density}
\end{equation}
where  $c_{m}=\Gamma(m/2)/\left(\pi^{m/2}\intop_{0}^{\text{\ensuremath{\infty}}}y^{m/2-1}\phi_{m}(y)\,dy\right)$
is a normalizing constant, the density generator $\phi_{m}(\cdot)$ is defined over $\mathbb{R}_{\geq0}$ and the parameter $\eta=\left(\boldsymbol{\mu},\mathrm{\mathbf{V}}^{-1}\right)$
$\in$ $\Xi$ = $\mathbb{R}^{m}\times\mathbb{M}_{+}^{m}$, the space
of $m\times m$ positive definite matrices, is unknown. To give a
few examples taken from \cite[Chap.~13]{Bilodeau1999}%Bilodeau \& Brenner (1999, Chap. 13)
, the
MVN has $\phi_{m}(y)=\exp\{-y/2\}$; the multiStudent with $\nu$
degrees of freedom has $\phi_{m}(y)=(1-y/\nu)^{-(m+\nu)/2}$; the
multivariate $\alpha$-th power exponential distribution has $\phi_{m}(y)=\textrm{exp}\left\{ -y^{\alpha}\right\} $\textcolor{blue}{.}
The multivariate slash-elliptical distribution with shape parameter
$\alpha$ \cite{Wang2006} has $\phi_{m}(y)=y^{-(m+\alpha)/2}\intop_{0}^{y/2}t^{(m+\alpha)/2}e^{-t}dt$.
A more elaborate version, the multivariate extended slash-elliptical
distribution has been introduced recently \cite{Rojas2014} to
fit heavy-tailed data. Additions to the catalogue of EC distributions
are periodically made in response to challenging new data sets. 

To exploit these statistical models in applications, a user needs
tools to help in selecting a distribution from this growing catalogue.
One such tool is a Goodness-of-Fit (GoF) test of the null hypothesis
that some data arise from a postulated EC density of the form (\ref{eq:1.1elliptical density}).
The goal of the present work is to develop such a GoF test having
attractive features.

An important feature of a GoF test is its power in detecting departures
from the null density. GoF tests can be
roughly divided into omnibus and directional. An omnibus test will
have power converging to one (i.e. be consistent) for any departure
from the null. A directional GoF test will be consistent for only some departures. At first glance, the omnibus property appears essential.
But equal power cannot be allocated to each departure and, for finite
samples, only some will have significant power (see \cite[Chap. 5]{Shorack1986}). Thus the interest for compromising
approaches offering some control over the power allocation, such as
the smooth test paradigm introduced by \cite{Neyman1937}. The smooth test
is directional along a set of $K$ departures, but the value of $K$
can be selected.  This creates a slider between directional and
omnibus GoF tests that yields an adaptive
form of the omnibus property which has been found, in many simulations,
to be very effective. This has brought \cite[p.~9]{Rayner1990} to recommend : ``\emph{Don't use those other methods\textendash use
a smooth test}!''. In its original form,
its main weakness is the lack of criteria to position $K$ on the
slider. To correct this, \cite{Ledwina1994} has introduced a version that selects $K$ in a data-driven fashion. The resulting
data-driven smooth test constitutes, power-wise, a significant improvement. This has prompted \cite{Kallenberg1997} to sharpen \cite{Rayner1990}'s recommendation into : ``\emph{use
a data-driven smooth test}!''. 

Another desirable feature of a GoF test is the ability to extract,
when the test rejects, some information regarding the aspects of the
null density contradicted by the data. As stated in \cite[p.~800]{Gelman1996} regarding the process of model checking: ``\emph{the
purpose of the checking is not simply to ``accept'' or ``reject'',
but rather to discover what aspects of the data are not being fit
well by the model}''. A GoF test that can provide such information
is said to have diagnostic (Dx) capabilities.
The smooth test has Dx capabilities, see \cite{Ducharme2016}.

Encouraged by these  features, much work
(\cite{Javitz1975,Thomas1979,Koziol1987,Kallenberg1997,Rayner2009,Thas2010} among others)
has been done to extend the smooth test paradigm. However, for multivariate
densities, few extensions have appeared. One hurdle comes from the fact that EC
distributions are closed under the group of affine-linear transformations
of the data. The invariance principle states that if
a  statistical problem is left invariant after a transformation, then its solution
should also be invariant under this transformation, otherwise
interpretability of the inference can be compromised. \cite[p.~469]{Henze2002} nicely summarizes the importance of this principle
: ``...\emph{ any proposal to use a non-invariant test.}.. \emph{must
come together with a special justification}''.  For the
case of the bivariate normal, \cite{Bogdan1999} has developed a data-driven
smooth test that cleverly combines the smooth test statistics for
univariate normality. But it is not invariant to rotations, so that an unscrupulous user could simply rotate
the data to reach a pre-specified conclusion.
Also her approach breaks down for general EC distributions because
a multivariate EC density does not always have the same density generator
as its marginals (the so-called inconsistency property, see \cite[p.~209]{Bilodeau1999}%Bilodeau \& Brenner, 1999, p. 209)
. \cite{Boulerice1997} develops a smooth GoF test invariant to rotations,
but their method is confined to data lying on an hypersphere.

Because of this scarcity, users facing the problem of assessing the
fit of a multivariate EC density may feel abandoned by statistical methodology. When the null density
is the MVN, numerous GoF tests exist (see \cite{Koziol1986,Romeu1993,Henze2002,Ozlem2016} and references therein).
However, GoF tests tailored to other EC distributions are almost non-existent 
(with the noticeable exceptions of \cite{Fragiadakis2011} and \cite{Ebner2012}) and a user must resort to a general-purpose GoF test, such as the multivariate version of the
Cramer von Mises approach (see \cite{Mcassey2013} for a short review), or ad hoc methods.

In this paper, we develop a smooth GoF test for EC densities of the
form (\ref{eq:1.1elliptical density}) with unknown parameter $\eta$,
that is  adaptively omnibus, has
Dx capabilities and is affine-linear invariant. As a bonus, the asymptotic
reference distribution is a standard $\chi^{2}$ and its power function
can be approximated by a sum of independent non-central $\chi^{2}$.
The approach is based on the ideas in \cite{Boulerice1997}
but adapted to the case of affine-linear transformations. Section
\ref{sec:Sec 2} adapts the smooth test paradigm to the case of a
general EC density and sets the stage for the several levels
of invariance required in our approach: problem invariance, test statistic invariance and Dx invariance. 
Section \ref{sec:Sec 3 the smooth test} makes explicit
the orthonormal basis on which our smooth test is based and derives
the associated test strategy by exploiting a variant of Rao's score test. Section
\ref{sec: Section 4} shows that the smooth
test statistic possesses a convenient explicit expression from which emerges a decomposition
into a sum of three invariant
and asymptotically independent $\chi^{2}$ distributed components,
which we refer to as the $\mathcal{Q}=\mathcal{UIR}$ decomposition.
It is explained that each component can provide
 interpretable Dx information about the aspects of the null
EC density contradicted by the data. It is also explained how
the omnibus/directional slider must be set up to preserve invariance
in all components of this decomposition. Section \ref{sec:Section 5 applications}
computes the $\mathcal{Q}=\mathcal{UIR}$ decomposition for the multivariate
normal (MVN) distribution and discusses how our smooth test extends
and relates to a number of proposals for this problem.  The R package ECGofTestDx can perform the necessary calculations. The case of
the bivariate Laplace is  also considered. Section \ref{sec:sec 6 Usefulness of Dx} reports
on an experiment that was conducted to see how the elements of the $\mathcal{Q}=\mathcal{UIR}$ decomposition
behave under a number of alternative densities somewhat representative
of what can be encountered in actual situations. It is seen that
the decomposition can indeed help in providing illuminating Dx information. Some
points requiring further research are collected in Section \ref{sec:Conclusions Sec 7}.
Proofs are confined to  Appendix \ref{sec: 8 Proofs-and-auxiliary}
and Appendix  \ref{sec:Appeb-B} gives the details for the computation of
the smooth test for the MVN when $m=2,3$ as well as MATHEMATICA commands
to extend in higher dimensions. Appendix  \ref{sec:Appeb-C} gives the smooth test for two other EC distribution, the  bivariate logistic and Pearson type II distributions.

\section{The Smooth Test Paradigm for Elliptical Distributions}\label{sec:Sec 2}

Let $\boldsymbol{X_{1}},\ldots,\boldsymbol{X}_{n}$ be independent
and identically distributed \emph{m}-dimensional observations
with density $\psi(\cdot)$. Consider the
problem of testing : 
\begin{equation}
H_{0}:\psi(\cdot)\in\mathcal{F}_{0}=\{f(\cdot;\eta),\eta\in\Xi\}\qquad vs\qquad H_{1}:\psi(\cdot)\notin\mathcal{F}_{0},\label{eq:2.1}
\end{equation}
where $f(\cdot;\eta)$ has the form (\ref{eq:1.1elliptical density})
with a given  $\phi_{m}(\cdot)$.
Here we consider the usual case where $\eta$ = $(\boldsymbol{\mu},\mathbf{V}^{-1})$ is unknown. We set the following :\medskip{}

\noindent \begin{flushleft}
\emph{Assumption} A : The support
of $f(\cdot;\eta)$ does not depend on $\eta$ and the mapping from
$\Xi$ to $\mathcal{F}_{0}$ is one-to-one. Moreover $f(\cdot;\eta)$
has a moment generating function so that all moments of $\boldsymbol{X}_{1}$
exist. Also $\mathbb{P}[\boldsymbol{X}_{1}=\boldsymbol{0}]=0.$
\par\end{flushleft}

\medskip{}

\noindent \begin{flushleft}
\emph{Assumption\/} B : $\hat{\eta}$
is an affine-equivariant estimator (in the sense of Definition 13.1
in \cite{Bilodeau1999}%Bilodeau \& Brenner, 1999
) of $\eta$ such that 
$\sqrt{n}(\hat{\eta}-\eta)=$ $\frac{1}{\sqrt{n}}\sum_{i=1}^{n}\boldsymbol{\ell}(\boldsymbol{X}_{i},\eta)+o_{p}(1)$,
where $\mathbb{E}_{0}(\boldsymbol{\ell}(\boldsymbol{X}_{1},\eta))=\mathbf{0}$
and the covariance matrix $\mathbb{V}_{0}(\boldsymbol{\ell}(\boldsymbol{X}_{1},\eta))$
is finite. 
\par\end{flushleft}

\medskip{}

Now  confine $\psi(\cdot)$ to the embedding family of densities 

\begin{align} 
\mathcal{G}= & \left\{ g(\cdot;\eta,h(\cdot))\left|\eta\in\Xi,h(\cdot)\in \mathbb{H} \textrm{ and }g(\cdot;\eta,\boldsymbol{0}(\cdot))=f(\cdot;\eta)\right.\right\} \label{eq:Definition G star}
\end{align}
where $\boldsymbol{0}(\cdot)$ is the zero function assumed to belong to a suitable space of functions $\mathbb{H}$. This allows to restate problem (\ref{eq:2.1}) as that of testing
\begin{equation}
H_{0}:h(\cdot)=\boldsymbol{0}(\cdot)\qquad vs\qquad H_{1}:h(\cdot)\neq\boldsymbol{0}(\cdot).\label{eq:2.2}
\end{equation}

Here  we adopt from \cite{Ducharme2001} embedding functions of the form 

\begin{equation}
g(\cdot;\eta,h(\cdot))=f(\cdot;\eta)\times\frac{(1+h(\cdot)-\mathbb{E}_{0}(h(\boldsymbol{X})))^{2}}{1+\Vert h(\cdot)-\mathbb{E}_{0}(h(\boldsymbol{X}))\Vert_{(f,\eta)}^{2}},\label{eq:2.3}
\end{equation}
where the subscript ``0''  refers to a statistical operator evaluated under $H_{0}$ and 
$\Vert h(\cdot)\Vert_{(f,\eta)}^{2}=\int h^{2}(\boldsymbol{x})f(\boldsymbol{x};\eta)d\boldsymbol{x} $. These are related to Hellinger's metric. 

The confinement to $\mathcal{G}$, and thus the choice of $\mathbb{H}$, affects the properties of our test. 
Regarding invariance, let $AL(m)$ be the group of affine-linear transformations on $\mathbb{R}^{m}$
with generic element $\gamma=$ $\gamma(\boldsymbol{x})=\mathbf{A(}\boldsymbol{x}+\boldsymbol{b})$,
represented by the pair $(\mathbf{A},\boldsymbol{b})$ where\emph{
$\mathbf{A}$} is a $m\times m$ non-singular real matrix and $\boldsymbol{b}\in$
$\mathbb{R}^{m}$. EC distributions are closed under these
transformations: the density of $\gamma(\boldsymbol{X})$ has the
form (\ref{eq:1.1elliptical density}) with its parameters transformed
from $\eta=(\boldsymbol{\mu},\mathbf{V}^{-1})$ to $\gamma^{*}(\eta)=$
$\mathrm{(\mathbf{A}}(\boldsymbol{\mu}+\boldsymbol{b}),(\mathbf{A}^{-1})^{T}\mathbf{V}^{-1}\mathbf{A}^{-1})$
\cite[p.~207]{Bilodeau1999}. %(Bilodeau \& Brenner, 1999, p. 207)
Our GoF test, and its Dx information, must  be
invariant to such transformations. To this end, let $L^{2}(f,\eta)=\{h(\cdot)\mid\Vert h(\cdot)\Vert_{(f,\eta)}^{2}<\infty\}$ and
 set $\mathbb{H} = \{h\in L^{2}(f,\eta)\,|\,\Vert h\circ\gamma(\cdot)\Vert_{(f,\eta)}<\infty,\forall\gamma\in AL(m)\}$.
With this choice, embedding family  (\ref{eq:Definition G star})  does not depend on the value of $\eta$ and is obviously closed under 
affine-linear transformations.This will be important to ensure invariance.
Note that $\mathbb{H}$ is a slight restriction of the Banach space  $L^{2}(f,\eta)$, so that almost
any reasonable $\psi(\cdot)$ can be written as (\ref{eq:2.3}). This is important for the omnibus property.

We now recall the following terminology. A subspace \emph{$\mathcal{H}$}
of $\mathbb{H}$ is said $AL(m)$-\emph{invariant} if $h(\cdot)$
$\in$ \emph{$\mathcal{H}$} implies $h\circ\gamma(\cdot)$ $\in$
$\mathcal{H}$ for all $\gamma$ $\in$ $AL(m)$. An $AL(m)$-invariant
subspace is further said to be \emph{irreducible} if it contains no
nontrivial $AL(m)$-invariant subspaces. For any $\eta\in\Xi$, let
$\left\langle \cdot,\cdot\right\rangle _{(f,\eta)}$ be the scalar
product associated with $\Vert\cdot\Vert_{(f,\eta)}$. $\mathbb{H}$
equipped with this scalar product is a Hilbert space denoted $\mathbb{H}_{(f,\eta)}$.
With these definitions at hand, a variant of the argument in \cite{Boulerice1997} shows that $\mathbb{H}_{(f,\eta)}$ can be decomposed
into a sequence of disjoint irreducible $AL(m)$-invariant subspaces
$\{\Pi_{k},k\geq0\}$, each of a finite dimension $d_{m}(k)$ that
will be made explicit in Theorem \ref{thm:3.1} below. Hence,
\begin{equation}
\mathbb{H}_{(f,\eta)}=\bigoplus_{k=0}^{\infty}\Pi_{k},\label{eq:2.4}
\end{equation}
where $\oplus$ is the direct sum operator with respect to $\left\langle \cdot,\cdot\right\rangle _{(f,\eta)}$
and $\Pi_{0}$ is the set of constant functions. For each $k\geq0$,
let $\{\pi_{k,j,\ell}(\cdot;\eta),(j,\ell)\in B_{k}\}$ be a complete orthonormal basis
(CONB) for $\Pi_{k}$. Here $B_{k}$ is a set of $d_{m}(k)$ pairs of integers that will be
explicited, along with a choice $\pi_{k,j,\ell}(\cdot;\eta)$
in Theorem \ref{thm:3.1}. Thus Span$\{\pi_{k,j,\ell}(\cdot;\eta),(j,\ell)\in B_{k}\}$
= $\Pi_{k}$  and for any $(j,\ell)$, $(j^{\prime},\ell^{\prime})$
$\in B_{k}$, $\left\langle \pi_{k,j,\ell}(\cdot;\eta),\pi_{k,j^{\prime},\ell^{\prime}}(\cdot;\eta)\right\rangle {}_{(f,\eta)}$
= $\delta_{jj^{\prime}}\delta_{\ell\ell^{\prime}}$ where $\delta$
is Kronecker's delta.
It follows from (\ref{eq:2.4}) that the whole
set of $\pi_{k,j,\ell}(\cdot;\eta)$ forms a CONB for $\mathbb{H}_{(f,\eta)}$
and that any $h(\cdot)$ $\in$ $\mathbb{H}_{(f,\eta)}$ can be written as
\begin{equation}
h(\cdot)=\sum_{k=0}^{\infty}\sum_{(j,\ell)\in B_{k}}\theta_{k,j,\ell}(\eta)\pi_{k,j,\ell}(\cdot;\eta),\label{eq:fourier expansion}
\end{equation}
where $\theta_{k,j,\ell}(\eta)$ = $\left\langle h(\cdot),\pi_{k,j,\ell}(\cdot;\eta)\right\rangle {}_{(f,\eta)}$. 
With (\ref{eq:fourier expansion}) the problem of testing (\ref{eq:2.2})
reduces to that of testing the nullity of all (for an omnibus test
within $\mathcal{G}$ ) or part (for a directional version) of
the $\theta_{k,j,\ell}(\eta)$. Focusing on the latter, fix $K$ and consider the subset of $\mathcal{G}$ that contains the  functions of the form
\begin{equation}
g_{K}(\cdot;\eta,\boldsymbol{\theta}(\eta))=f(\cdot;\eta)\times\frac{\left(1+\sum_{k=1}^{K}\sum_{(j,\ell)\in B_{k}}\theta_{k,j,\ell}(\eta)\pi_{k,j,\ell}(\cdot;\eta)\right)^{2}}{1+\sum_{k=1}^{K}\sum_{(j,\ell)\in B_{k}}\theta_{k,j,\ell}^{2}(\eta)}.\label{eq:2.5}
\end{equation}
Further  confining $\psi(\cdot)$ to this subset reduces the problem of testing (\ref{eq:2.1}) to that
of testing 
\begin{equation}
H_{0}:\boldsymbol{\theta}(\eta)=\boldsymbol{0}\qquad vs\qquad H_{1}:\boldsymbol{\theta}(\eta)\neq\boldsymbol{0},\label{eq: the final null hypothesis}
\end{equation}
where $\boldsymbol{\theta}(\eta)=\{\theta_{k,j,\ell}(\eta),1\leq k\leq K,(j,\ell)\in B_{k}\}$.
Under (\ref{eq:2.5}), it is easy to see that if $\boldsymbol{\theta}(\eta)$
is small, the dominant term in the maximum likelihood estimator (mle)
of $\theta_{k,j,l}($$\eta)$ is $n^{-1}\sum_{i=1}^{n}\pi_{k,j,\ell}(\boldsymbol{X}_{i};\eta)$.
Hence, it makes statistical sense (for another justification, see
\cite{Klar2000}) to base a test of (\ref{eq: the final null hypothesis})
on
\begin{equation}
\bar{\boldsymbol{\pi}}(\hat{\eta})=\left(n^{-1}\sum_{i=1}^{n}\pi_{k,j,\ell}(\boldsymbol{X}_{i};\hat{\eta}),1\leq k\leq K,(j,\ell)\in B_{k}\right)^{T}.\label{eq:2.6}
\end{equation}

If the $\pi_{k,j,\ell}(\cdot;\cdot)$ are continuously differentiable
in $\eta$, then Theorem 2.1 a) of \cite{Klar2000} ensures the asymptotic
normality (with asymptotic expectation $\mathbf{0}$) of $\sqrt{n}\bar{\boldsymbol{\pi}}(\hat{\eta})$
under $H_{0}$. The asymptotic covariance matrix depends on the choice
of $\hat{\eta}$. In particular, if $f(\cdot;\eta)$ satisfies the
conditions in \cite[p.~121]{Ferguson1996}, then $\hat{\eta}$ can be taken
as the mle of $\eta$, which satisfies \emph{Assumption} B with $\mathbb{V}_{0}(\mathbf{\ell}(\boldsymbol{X}_{i},\eta))=\boldsymbol{\mathcal{J}}_{\eta}^{-1}$,
where $\boldsymbol{\mathcal{J}}_{\eta}$ is Fisher's information for
$\eta$ under $H_{0}$. If $\mathbf{J}_{\eta}$ denotes the matrix
with elements $\mathbb{C\mathrm{ov}}_{0}\left(\pi_{k,j,\ell}(\boldsymbol{X}_{i};\eta),\partial\log f(\boldsymbol{X}_{i};\eta)/\partial\eta_{a}\right)$
where $\eta_{a}$ is a $a$-th component of $\eta$, then from Theorem
2.1 c) of \cite{Klar2000}, this covariance matrix is $\mathbf{I}_{\nu_{\mathcal{Q}}}-\mathbf{J}_{\eta}\boldsymbol{\mathcal{J}}_{\eta}^{-1}\mathbf{J}_{\eta}^{T}$
where $\mathbf{I}_{v}$ is the identity matrix of order $\nu$  and
$\nu_{\mathcal{Q}}$ = $\sum_{k=1}^{K}d_{m}(k)$. Assume this matrix invertible. Then test statistic 
\begin{equation}
\mathcal{Q}_{K}=\mathcal{Q}_{K}(\boldsymbol{X},\hat{\eta})=n\bar{\boldsymbol{\pi}}^{T}(\hat{\eta})(\mathbf{I}_{v_{\mathcal{Q}}}-\mathbf{J}_{\hat{\eta}}\boldsymbol{\mathcal{J}}_{\hat{\eta}}^{-1}\mathbf{J}_{\hat{\eta}}^{T})^{-1}\bar{\boldsymbol{\pi}}(\hat{\eta})\label{eq:2.7}
\end{equation}
is under $H_{0}$ asymptotically $\chi_{\nu_{\mathcal{Q}}}^{2}$.
We refer to $\mathcal{Q}_{K}$ as a \emph{global} test statistic for
$H_{0}$. As a by-product of the CONB introduced in the next section,
an explicit expression for this test statistic is available (see (\ref{eq:alternate expression for the test statistic})
and (\ref{eq: U I R decomp})).

For this to make statistical sense requires $AL(m)$-invariance
of version (\ref{eq: the final null hypothesis}) of problem
(\ref{eq:2.1}) and $AL(m)$-invariance of test statistic $\mathcal{Q}_{K}(\boldsymbol{X},\hat{\eta})$.
The latter will be tackled in Sections \ref{sec:Sec 3 the smooth test}
and \ref{sec: Section 4}. The former requires that, after transforming
into $\gamma(X)$, the new null $\boldsymbol{\theta}(\gamma^{*}(\eta))=\boldsymbol{0}$
holds if and only if $\boldsymbol{\theta}(\eta)=\boldsymbol{0}.$
Because of the properties of the $\Pi_{k}$ in (\ref{eq:2.4}), a
condition for this is that all $\pi_{k,j,\ell}(\cdot;\eta),(j,\ell)\in B_{k}$
must appear in (\ref{eq:2.5}) with no $\theta_{k,j,\ell}(\eta)$
structurally set to zero (i.e. no term $\pi_{k,j,\ell}(\cdot;\eta)$
systematically excluded). As will be seen
in Remark \ref{rem: remark choix de lamnda}, some components of $\bar{\boldsymbol{\pi}}(\hat{\eta})$
may systematically vanish, i.e. $\pi_{k,j,\ell}(\boldsymbol{X}_{i};\hat{\eta})\equiv0$.
Then all other terms in $\left\{ \pi_{k,j,\ell}(\cdot;\eta),(j,\ell)\in B_{k}\right\} $
must be dropped from (\ref{eq:2.6}) to preserve problem invariance.
As will be seen in Section \ref{sec: Section 4}, the same condition,
coupled with our choice for the $\pi_{k,j,\ell}(\cdot;\eta)$ also
ensures $AL(m)$-invariance of $\mathcal{Q}_{K}(\boldsymbol{X},\hat{\eta})$.

\begin{rem}
\label{rem:Remark 2} Proving the unicity of the mle $\hat{\eta}$ can be difficult. 
Some conditions are given in \cite{Kent1991}. The
verification of regularity conditions ensuring \emph{Assumption} B (e.g. \cite{Ferguson1996}) 
is also tedious. Some work has been done in 
\cite{Hassan2005} for particular cases of the $\alpha$-th
power exponential distribution. Other estimators could in principle
be used but the matrix in (\ref{eq:2.7})
becomes more complicated. Theorem 2.1.a) in \cite{Klar2000} gives its
general expression without exploiting the particular structure of
elliptical distributions. 
\end{rem}
\begin{rem}
\label{rem:Remark on power}In view of (\ref{eq:2.7}), the power
function for a fixed $\psi(\cdot)$ can, under mild assumptions
(see \cite{Inglot1994}), be approximated by 
\begin{alignat}{1}
\mathbb{P}_{\psi}[\mathcal{Q}_{K}>c] & =P\left[\sum_{k=1}^{K}\sum_{(j,\ell)\in B_{k}}\lambda_{k,j,\ell}\chi_{1}^{2}(n\nu_{k,j,\ell}^{2})>c\right]+O(n^{-1/2}),\label{eq:approximate power function}
\end{alignat}
where $\nu_{k,j,\ell}$ are functions of $\int\pi_{k,j,\ell}(\boldsymbol{x};\eta)\psi(\boldsymbol{x})d\boldsymbol{x}$
and $\lambda_{k,j,\ell}$ are the eigenvalues of a complicated matrix (see
\cite[Sec.~3.2]{Ducharme2016} for details in the case where
$\mathbf{J}_{\eta}=\mathbf{0}$, as in Section \ref{subsec:Sous section sur la MVN}).
When $\psi(\cdot) $ is of the form (\ref{eq:2.7}), 
$\nu_{k,j,\ell}\simeq O(\theta_{k,j,\ell}(\eta))$. If one component
of $\boldsymbol{\theta}(\eta)$ in (\ref{eq: the final null hypothesis})
differs from 0, the test will be consistent. This offers a 
handle on  balancing the directional/omnibus slider. On the one
hand, power is lost when, for  $k > K$, one $\theta_{k,j,\ell}(\eta)$ 
is large. On the other hand, because $\theta_{k,j,\ell}(\eta)\rightarrow 0$
with $k$, a large $K$ could add terms close to a $\chi_{1}^{2}(0)$
in (\ref{eq:approximate power function}), leading to
power dilution. To properly set a balance, some knowledge about
$\psi(\cdot)$ is needed; otherwise see Remark \ref{rem:remarque sur omnibusite}
for choosing $K$ in a data-driven fashion. But these remarks
pertain to the $\theta_{k,j,\ell}(\eta)$ : a compounding difficulty is that some $\lambda_{k,j,\ell}$
may also be small, with the corresponding $\lambda_{k,j,\ell}\chi_{1}^{2}(n\nu_{k,j,\ell}^{2})$
stochastically different from a $\chi_{1}^{2}(0)$, which again affects power, in some cases rendering the test biased for
small samples; see the $\alpha-$th power exponential
case in Section \ref{sec:sec 6 Usefulness of Dx}. 
\end{rem}
\begin{rem}
\label{rem:Remark 2-2} The above test can in principle be extended
to the case where the density generator $\phi_{m}(\cdot)$ in (\ref{eq:1.1elliptical density})
involves a shape parameter, such as with the $\alpha$-th power exponential distribution.
We do not pursue this further here as the necessary developments are
beyond the scope of the paper. 
\end{rem}
\begin{rem}
\label{rem:Remark 5 on Monte Carlo} Convergence toward the 
asymptotic $\chi^{2}$  distribution can be slow. A benefit of invariance
is that better approximations can be obtained
by Monte Carlo resampling from any convenient $f(\cdot;\eta)$ distribution
in $\mathcal{F}_{0}$.
\end{rem}

\section{The smooth test statistic}\label{sec:Sec 3 the smooth test}

\subsection{Construction of the basis}\label{subsec:Sec3.1}

Suppose that \emph{$\boldsymbol{X}$} has a density of the form (\ref{eq:1.1elliptical density}).
The representation in \cite{Cambanis1981} 
states that the random vector \emph{$\boldsymbol{Y}$} = $\mathbf{V}^{-1/2}(\boldsymbol{X}-\boldsymbol{\mu})$
has a spherical distribution (i.e. density (\ref{eq:1.1elliptical density})
with $\eta$ = $(\mathbf{0},\mathbf{I}_{m})$) with stochastic representation 
\begin{equation}
\boldsymbol{Y}=R\boldsymbol{U},\label{eq:3.1}
\end{equation}
where \emph{$R$} = $R(\boldsymbol{X})$ = $\Vert\mathbf{V}^{-1/2}(\boldsymbol{X}-\boldsymbol{\mu})\Vert$
is referred to as the radius, independent of \emph{$\boldsymbol{U}$}
= $\boldsymbol{U}(\boldsymbol{X})$ = $\boldsymbol{Y}/\left\Vert \boldsymbol{Y}\right\Vert \sim U(\Omega_{m})$,
the uniform distribution on the unit sphere $\Omega_{m}$. Here $\left\Vert \cdot\right\Vert $
is the Euclidean norm. The density of \emph{$\boldsymbol{Y}$} can
thus be parametrized in polar coordinates as
\begin{equation}
f_{R,\boldsymbol{U}}(r,\boldsymbol{u})=c_{m}\phi_{m}(r^{2})r^{m-1}\,dr\,d\omega_{m}(\boldsymbol{u}),\label{eq:3.2}
\end{equation}
where $d\omega_{m}(\boldsymbol{u})$ denotes the area element on $\Omega_{m}$
with $\omega_{m}(\Omega_{m})$ = $2\pi^{m/2}/\Gamma(m/2)$. This mapping
from \emph{$\boldsymbol{X}$} $\in$ $\mathbb{R}^{m}$ to $(R,\boldsymbol{U})$
$\in$ $\mathbb{R}^{+}\times\Omega_{m}$ depends on $\eta$ but $f_{R,\boldsymbol{U}}(r,\boldsymbol{u})$
is invariant to the choice of $\mathbf{V}^{-1/2}$. 

A CONB for $\mathbb{H}_{(f,\eta)}$ will be obtained by combining
elements of CONBs associated with the distributions of \emph{R} and
\emph{$\boldsymbol{U}$,} so that $\pi_{k,j,\ell}(\boldsymbol{X};\eta)$
can be written as $\pi_{k,j,\ell}(R(\boldsymbol{X}),\boldsymbol{U}(\boldsymbol{X}))$.

First consider \emph{$\boldsymbol{U}$}. Let $\mathcal{O}(m)$ be
the group of rotations on $\mathbb{R}^{m}$. From \cite[p.~17]{Helgason1984}, the space $E_{m}(k)$ of hyperspherical 
harmonics of degree \emph{$k$} in \emph{$m$}
dimensions is an irreducible $\mathcal{O}(m)$\textendash invariant
subspace of the space of homogeneous polynomials of degree \emph{k}
on $\Omega_{m}$ and, from \cite[Thm.~5.12, p.~81]{Axler2001}%Axler, Bourdon \& Ramey (2001, Thm. 5.12, p. 81)
, the space of square integrable functions on $\Omega_{m}$
can be decomposed as :

\[
L^{2}(\Omega_{m})=\bigoplus_{k=0}^{\infty}E_{m}(k).
\]
Again from \cite[Prop.~5.8, p.~78]{Axler2001}%Axler, Bourdon \& Ramey (2001, Prop. 5.8, p.78)
, 
\begin{equation}
\dim(E_{m}(k))=e_{m}(k)=\begin{cases}
1 & \mathrm{if\,\,}k=0\\
m & \mathrm{if\,\,}k=1\\
C_{m-1}^{m+k-1}-C_{m-1}^{m+k-3} & \mathrm{if\,\,}k\geq2
\end{cases}.\label{eq:3.3}
\end{equation}

Let $\{\Psi_{k,\ell}(\boldsymbol{u}),\ell=1,\ldots,e_{m}(k)\}$
be a CONB for $E_{m}(k)$ with respect to the scalar product $\left\langle h_{1},h_{2}\right\rangle _{d\omega_{m}}=$
$(\omega_{m}(\Omega_{m}))^{-1}\intop_{\Omega_{m}}h_{1}(\boldsymbol{u})h_{2}(\boldsymbol{u})d\omega_{m}(\boldsymbol{u})$.
When $m=2$, writing $\boldsymbol{u}^{T}=(\mu_{1},\mu_{2})=(\textrm{cos}\,\theta,\textrm{sin}\,\theta)$ 
we have $\Psi_{k,1}(\mu_{1},\mu_{2})=\sqrt{2}\,\textrm{cos}(k\theta),$
$\Psi_{k,2}(\mu_{1},\mu_{2})=\sqrt{2}\,\textrm{sin}(k\theta)$. When
\emph{m} = 3,
$\Psi_{k,\ell}(\mu_{1,}\mu_{2},\mu_{3})$  $= P_{k-j}^{j}(\mu_{1})\Psi_{k,j^{\prime}}(\mu_{2},\mu_{3})$
where $P_{k-j}^{j}(\cdot)$ is the associated Legendre function of
the first kind of order $j$ and degree $k-j$, for $j=0,...,k$,
$j^{\prime}=1,\min\{j+1,2\}$. This is a particular case of a recurrence
formula, explained in \cite{Boulerice1997},
that generates the hyperspherical harmonics in dimension $m$ from those in
lower dimensions. Such calculations require symbolic manipulations. In particular, \cite[Appendix~B]{Axler2001} %Axler, Bourdon \& Ramey (2001, Appendix B)
give details about the MATHEMATICA program \texttt{HFT10.m} that computes
 $\{\Psi_{k,\ell}(\boldsymbol{u}), \ell=1,\ldots,e_{m}(k)\}$
for any $m$. To facilitate the application of the methods of the
paper, the instructions  to generate $\Psi_{k,\ell}(\boldsymbol{u})$
with this program are detailed in our Appendix \ref{sec:Appeb-B}.

Next, we look at the radius\emph{ R}. For each $i\geq0$, consider
the set of functions $\{s_{j,i}(r),j\geq0\}$ where $s_{j,i}(\cdot)$
is a polynomial of degree \emph{j} in $r^{2}$ satisfying
\begin{equation}
\omega_{m}(\Omega_{m})c_{m}\intop_{0}^{\infty}s_{j,i}(r)s_{j^{\prime},i}(r)\phi_{m}(r^{2})r^{m-1+2i}dr=\delta_{j,j^{\prime}}.\label{eq:3.4}
\end{equation}
To compute these polynomials, we apply the method described in
Section 3 in \cite{Boulerice1997}, but using the scalar product
(\ref{eq:3.4}). For $i\geq0$, let $\mu_{j,i}=\mathbb{E}_{0}(R^{2(j+i)})$.
Write $\boldsymbol{\mu}_{j,i}$ = $(\mu_{j,i},\mu_{j+1,i},\ldots,\mu_{2j-1,i})^{T}$
and 

\[
\mathbf{M}_{j,i}=\left(\begin{array}{ccc}
\mu_{0,i} & \cdots & \mu_{j-1,i}\\
\vdots & \ddots & \vdots\\
\mu_{j-1,i} & \cdots & \mu_{2j-2,i}
\end{array}\right).
\]
Set these quantities to 0 when \emph{j} = 0. Using the argument
 leading to (4.4) of \cite{Boulerice1997}, we get
\begin{equation}
s_{j,i}(r)=\frac{r^{2j}-(1,r^{2},\ldots,r^{2j-2})\mathbf{M}_{j,i}^{-1}\boldsymbol{\mu}_{j,i}}{\sqrt{\mu_{2j,i}-\boldsymbol{\mu}_{j,i}^{T}\mathbf{M}_{j,i}^{-1}\boldsymbol{\mu}_{j,i}}}.\label{eq:3.5}
\end{equation}
The MATHEMATICA commands for these functions in the important
case of the MVN distribution are in Appendix  \ref{sec:Appeb-B}. The following
theorem, whose proof is given in Appendix \ref{subsec:Proof-of-Theorem 3.1},
explains how to construct a CONB (i.e. the $\pi_{k,j,\ell}(\cdot;\cdot)$
of the previous section) associated with an $f(\cdot;\eta)$ of the
form (\ref{eq:1.1elliptical density}).
\begin{thm}
\label{thm:3.1} Fix $\eta\in\Xi$. Let $\mathcal{P}_{k}$ be the
space of polynomials of degree k in $\boldsymbol{x}\in\mathbb{R}^{m}$.
Let $\Pi_{0}$ = $\mathcal{P}_{0}$ and for $k\geq1$, let $\Pi_{k}$
be the subspace of \textup{$\mathcal{P}_{k}$} such that $p(\cdot)$
$\in$ \textup{$\Pi_{k}$} if 
\begin{equation}
\intop_{\mathbb{R}^{m}}p(\boldsymbol{x})p_{k^{\prime}}(\gamma(\boldsymbol{x}))f(\boldsymbol{x};\eta)d\boldsymbol{x}=0\label{eq:3.6}
\end{equation}
for all $p_{k^{\prime}}(\cdot)$ $\in$ \textup{$\Pi_{k^{\prime}}$,
0 $\leq k^{\prime}<k$ }\textup{\emph{ and}}\textup{ $\gamma$ $\in$$AL(m).$
}\textup{\emph{Then,}}

\smallskip{}
i) The subspaces \textup{$\Pi_{k}$} are orthogonal with respect to
$\left\langle \cdot,\cdot\right\rangle {}_{(f,\eta)}$. Moreover, 

\[
\mathcal{P}_{k}=\Pi_{0}\oplus\Pi_{1}\oplus\cdots\oplus\Pi_{k}
\]
and
\begin{equation}
\mathbb{H}_{(f,\eta)}=\bigoplus_{k=0}^{\infty}\Pi_{k}.\label{eq: proof of consistency}
\end{equation}

\smallskip{}
ii) The subspace $\Pi_{k}$ is an irreducible AL(m)-invariant subspace
of the space $\mathcal{P}_{k}$. The polynomials 
\begin{equation}
\pi_{k,j,\ell}(R(\boldsymbol{x}),\boldsymbol{U}(\boldsymbol{x}))=R(\boldsymbol{x})^{k-2j}s_{j,k-2j}(R(\boldsymbol{x}))\Psi_{k-2j,\ell}(\boldsymbol{U}(\boldsymbol{x})),\label{eq:Forme explicite des polyn=0000F4mes Pi_k,j,l}
\end{equation}
with $(j,\ell)\in B_{k}$ = $\{j=0,\ldots,[k/2],\ell=1,\ldots,e_{m}(k-2j)\}$
form a CONB for \textup{$\Pi_{k}$ }\textup{\emph{with respect to
}}$\left\langle \cdot,\cdot\right\rangle {}_{(f,\eta)}$. Here $[z]$
denotes the integer part of z. Moreover \textup{$\Pi_{k}$} is of
dimension $d_{m}(k)$ = $\sum_{j=0}^{[k/2]}e_{m}(k-2j)$ = $C_{k}^{m+k-1}$.
Finally each \textup{$\Pi_{k}$} can be further decomposed as

\begin{equation}
\Pi_{k}=\bigoplus_{j=0}^{[k/2]}\Pi_{k,k-2j},\label{eq: the decomposition of Q-UIR decomposition}
\end{equation}
where $\Pi_{k,j}$ is spanned by $\{\pi_{k,j,\ell}(R(\boldsymbol{x}),\boldsymbol{U}(\boldsymbol{x})),\ell=1,\ldots,e_{m}(k-2j)\}$
and is an irreducible $\mathcal{O}(m)$-invariant subspace of $\Pi_{k}$
whose dimension is $e_{m}(k-2j)$.
\end{thm}
Beside giving the  expression (\ref{eq:Forme explicite des polyn=0000F4mes Pi_k,j,l}) for the $\pi_{k,j,\ell}(\cdot;\eta)$,
this theorem is crucial for at least two other points. First, (\ref{eq: proof of consistency})
ensures that an embedding family of the form (\ref{eq:2.5}) with
$K$ properly chosen and with (\ref{eq:Forme explicite des polyn=0000F4mes Pi_k,j,l})
as the elements in the basis, can approximate almost any alternatives. Thus the test
can be made adaptively omnibus; see Remark \ref{rem:remarque sur omnibusite}.
Second, as will be seen in Section \ref{sec: Section 4}, (\ref{eq: the decomposition of Q-UIR decomposition})
provides conditions ensuring the invariance of Dx tools. Note also
that the differentiability of the $\pi_{k,j,\ell}(\cdot;\eta)$ required
to apply \cite{Klar2000}'s results, follows from the structure in
(\ref{eq:Forme explicite des polyn=0000F4mes Pi_k,j,l}). 

\subsection{The test statistic}\label{subsec:The-test-statistic 3.2}

The smooth test for null hypothesis (\ref{eq:2.1}) can now be 
explicited. Estimate $\eta$ by the mle $\hat{\eta}=(\hat{\boldsymbol{\mu}},\hat{\mathbf{V}}^{-1})$
with estimated information matrix $\boldsymbol{\mathcal{J}}_{\hat{\eta}}$.
Set $\hat{\boldsymbol{Y}}_{i}$ = $\hat{\mathbf{V}}^{-1/2}(\boldsymbol{X}_{i}-\hat{\boldsymbol{\mu}})$
and compute $\hat{R}_{i}$ = $\Vert\hat{\boldsymbol{Y}}_{i}\Vert$,
$\hat{\boldsymbol{U}}_{i}$ = $\hat{\boldsymbol{Y}}_{i}/\Vert\hat{\boldsymbol{Y}}_{i}\Vert$.
Inject these into the $\pi_{k,j,\ell}(\hat{R}_{i},\hat{\boldsymbol{U}}_{i})$
to compute $\bar{\pi}_{k,j,\ell}=n^{-1}\sum_{i=1}^{n}\pi_{k,j,\ell}(\hat{R}_{i},\hat{\boldsymbol{U}}_{i})$.
Group these into vector $\bar{\boldsymbol{\pi}}(\hat{\eta})$ of (\ref{eq:2.6})
and compute $\mathbf{J}_{\hat{\eta}}$ . Test statistic (\ref{eq:2.7})
can be obtained. Null hypothesis (\ref{eq:2.1})
is rejected at approximate level $\alpha$ if $\mathcal{Q}_{K}$ is
greater than the $1-\alpha$-th quantile of the reference distribution,
e.g. the Monte Carlo approximation or the $\chi_{\nu_{\mathcal{Q}}}^{2}$
distribution with \emph{$\nu_{\mathcal{Q}}$} = $\sum_{k=k_{min}}^{K}C_{k}^{m+k-1}$.

\smallskip{}

\begin{rem}
\label{rem: remark choix de lamnda}As stated in Section \ref{sec:Sec 2},
to ensure invariance care must be taken in selecting the elements
of $\bar{\boldsymbol{\pi}}(\hat{\eta})$. First, any
$\pi_{k,j,\ell}(\cdot,\cdot)$ that is a linear combination of some
of the $\partial\log f(\boldsymbol{X};\eta)/\partial\eta_{a}$  must be excluded because
they will contribute nothing to the test statistic. Next, irreducibility
of $\Pi_{k}$ dictates that to retain invariance, when $\pi_{k,j,\ell}(\cdot,\cdot)$
has been excluded, then all other $\{\pi_{k,j^{\prime},\ell^{\prime}}(\cdot,\cdot)$
$(j^{\prime},\ell^{\prime})\in B_{k}\}$ must also be excluded.
 \end{rem}

\smallskip{}

\begin{rem}
\label{rem:remarque sur omnibusite} One
difficulty in applying any smooth test resides in selecting the value
of $K$. For univariate data, \cite{Ledwina1994}'s data-driven smooth test first chooses an integer $d(n)$, performs
a selection procedure to get a ``good'' $\hat{K}\in[1,...,d(n)]$,
and finally computes the associated test statistic. The framework where $d(n)\rightarrow\infty$
leads to an adaptively omnibus test and is referred to as the\emph{
}infinite horizon case \cite[p.~102]{Thas2010}. The
available theory about the rate of divergence of $d(n)$ is impressive
but hardly translates into a precise value, as it is expressed
in terms of $o(\cdot)$. Thus $d(n)$ is
in practice fixed by external considerations,
the finite horizon framework. Fortunately, simulations show that
the power of the data-driven smooth test stabilizes rapidly as $d(n)$
increases.Hence in practice both frameworks lead to the same modus operandi 
as long as $d(n)$ is not too small.

The methodology that derives from Theorem \ref{thm:3.1} allows to
compute the smooth test up to any desired $d(n)$, so it is in principle
possible to develop both infinite and finite horizon data-driven tests
in our context. Here we only sketch a simple invariant adaptation of \cite{Ledwina1994}'s approach and focus on the finite horizon case.
Suppose $d(n)$ is given. Define 
\begin{flalign}
\hat{K} & =\min\{k:1\leq k\leq d(n),\mathcal{Q}_{k}-\left[\sum_{j=1}^{k}\textrm{Card}(B_{j})\right]\log(n)\label{eq: Choix de K_hat}\\
 & \geq\mathcal{Q}_{\ell}-\left[\sum_{j=1}^{\ell}\textrm{Card}(B_{j})\right]\log(n),\ell=1,...,d(n)\}.\nonumber 
\end{flalign}
The asymptotic reference distribution of
test statistic $\mathcal{Q}_{\hat{K}}$ is a $\chi^{2}$ with $C_{1}^{m}$
degrees of freedom. As $d(n)$ is allowed to increase, $\mathcal{Q}_{\hat{K}}$
spreads its power in a data-driven fashion over an increasing number
of directions. The exploration of this and other scenarios is
left for future work ; see Section \ref{sec:Conclusions Sec 7}.

\end{rem}

\section{Invariance and the $\mathcal{Q}=\mathcal{UIR}$ decomposition}\label{sec: Section 4}

We now take a closer look at matrix $(\mathbf{I}_{v_{\mathcal{Q}}}-\mathbf{J}_{\eta}\boldsymbol{\mathcal{J}}_{\eta}^{-1}\mathbf{J}_{\eta}^{T})^{-1}$
in (\ref{eq:2.7}).
It is explained in Section \ref{subsec:8.2 Derivation-of-the decomposition}
that this is a matrix of constants and that after permuting the $\pi_{k,j,\ell}(\cdot,\cdot)$
in (\ref{eq:2.6}) according to the values of $k-2j$, test statistic
(\ref{eq:2.7}) can be conveniently written as (\ref{eq: final value of the test statistic}),
which leads to 
\begin{align}
\mathcal{Q}_{K} & =\mathcal{Q}_{K}(\boldsymbol{X},\hat{\eta})=\mathcal{U}_{K}+\mathcal{I}_{K}+\mathcal{R}_{K},\label{eq:alternate expression for the test statistic}
\end{align}
which we refer to as the $\mathcal{Q}=\mathcal{UIR}$ decomposition,
where
\begin{alignat}{1}
\mathcal{U}_{K} & =n\left\Vert \boldsymbol{\bar{\pi}}_{\mathcal{U}}\right\Vert ^{2},\nonumber \\
\mathcal{I}_{K} & =n\left\{ \left\Vert \boldsymbol{\bar{\pi}}_{\mathcal{I},1}\right\Vert ^{2}+d_{1}\textrm{tr}(\mathbf{c}_{1}\mathbf{c}_{1}^{T}\bar{\mathfrak{I}}_{1}\bar{\mathfrak{I}}_{1}^{T})+\left\Vert \boldsymbol{\bar{\pi}}_{\mathcal{I},2}\right\Vert ^{2}+d_{2}\textrm{tr}(\mathbf{c}_{2}\mathbf{c}_{2}^{T}\bar{\mathfrak{I}}_{2}\bar{\mathfrak{I}}_{2}^{T})+\left\Vert \boldsymbol{\bar{\pi}}_{\mathcal{I},3}\right\Vert ^{2}\right\} ,\label{eq: U I R decomp}\\
\mathcal{R}_{K} & =n\left\{ \left\Vert \boldsymbol{\bar{\pi}}_{\mathcal{R}}\right\Vert ^{2}+d_{0}\textrm{tr}(\mathbf{c}_{0}\mathbf{c}_{0}^{T}\boldsymbol{\bar{\pi}}_{\mathcal{R}}\boldsymbol{\bar{\pi}}_{\mathcal{R}}^{T})\right\} ,\nonumber 
\end{alignat}
are asymptotically independent
$\chi^{2}$ with degrees of freedom $\nu_{\mathcal{U}},\nu_{\mathcal{I}},\nu_{\mathcal{R}}$
given in (\ref{subsec:8.2 Derivation-of-the decomposition}) under $H_{0}$. Note
that, beyond the moments of $R^{2}$ required for the $s_{j,i}(\cdot)$,
the computation of this $\mathcal{Q}=\mathcal{UIR}$ decomposition
conveniently requires only the $2+2\left[K/2\right]+\left[(K+1)/2\right]$
quantities described in Section \ref{subsec:8.2 Derivation-of-the decomposition}
to get the vectors $\sqrt{d_{0}}\text{\ensuremath{\mathbf{c}}}_{0}$,$\sqrt{d_{1}}\text{\ensuremath{\mathbf{c}}}_{1}$
and $\sqrt{d_{2}}\text{\ensuremath{\mathbf{c}}}_{2}$.

The elements $\pi_{k,0,\ell}(r,\boldsymbol{u})$ in $\mathcal{U}_{K}$
are proportional to $r^{k}\Psi_{k,\ell}(\boldsymbol{u})$. The components
of $\boldsymbol{\bar{\pi}}_{\mathcal{U}}$ are thus weighted averages
of polynomials in $\hat{\boldsymbol{U}}_{i}$.
Note that up to a constant, the $\pi_{k,0,\ell}(r,\boldsymbol{u})$
are identical for any EC distribution and thus cannot be related
to the distribution of \emph{$R^{2}$}. Hence $\mathcal{U}_{K}$ 
serves to detect departures from the uniformity of \emph{$\boldsymbol{U}$}.
 When such departures are detected, the
true density of the data may not be constant on ellipses centered at
$\boldsymbol{\mu}$ and a model incorporating this feature should
be seeked, e.g. a distribution where $\boldsymbol{U}$ possesses
a more complex density on $\Omega_{m}$.

Similarly, because $\Psi_{0,1}(\boldsymbol{u})\equiv1$, the elements
of $\mathcal{R}_{K}$ are $\pi_{k,k/2,1}(r,\boldsymbol{u})$ = $s_{k,0}(r)$
with $k$ even. Hence this component,
which under $H_{0}$ is asymptotically $\chi_{\nu_{\mathcal{R}}}^{2}$,
serves to detect departures from the distribution of the ``radius''
$R$. These departures may then be identified,
visually or otherwise, and the null model corrected accordingly. $\mathcal{R}_{K}$
can also be used when, as in \cite{Kariya1995},  it is desired to have a test whose 
power is directed toward alternatives that are also
elliptic about the unknown $\boldsymbol{\mu}$.

Finally, consider the elements in $\mathcal{I}_{K}$. These are
products of polynomials in \emph{R} with hyperspherical harmonics in \emph{$\boldsymbol{U}$}.
Thus $\mathcal{I}_{K}$, which is approximately $\chi_{\nu_{\mathcal{I}}}^{2}$,
detects correlations between\emph{ $\boldsymbol{U}$} and \emph{R}.
When this occurs, the structure of the true density is complicated
and a user could look at more involved densities, e.g. skew-densities
of some sort. 

Summing up, each component of the $\mathcal{Q}=\mathcal{UIR}$ decomposition
can detect a meaningful, in terms of stochastic representation (\ref{eq:3.1}),
type of departure from the null density.

$AL(m)$-invariance of $\mathcal{Q}_{K}$ and each
of its components follows from the following considerations that hold
provided all $\pi_{k,j,\ell}(\cdot,\cdot)$ spanning each $\Pi_{k}$,
and thus each $\Pi_{k,k-2j}$ in (\ref{eq: the decomposition of Q-UIR decomposition}),
appear in (\ref{eq:2.6}). For any $\gamma=(\mathbf{A},\boldsymbol{b})\in AL(m)$,
write $\boldsymbol{X}_{i}^{*}=\gamma(\mathbf{X}_{i})$, $\hat{\eta}^{*}=\gamma^{*}(\hat{\eta})=(\mathbf{\hat{\boldsymbol{\mu}}}_{*},\mathbf{\hat{V}}_{*}^{-1})$
defined in Section \ref{sec:Sec 2}. Also define $\hat{\boldsymbol{Y}}_{i}^{*}=\hat{\mathbf{V}}_{*}^{-1/2}(\boldsymbol{X}_{i}^{*}-\hat{\boldsymbol{\mu}}_{*})$
and similarly for $\hat{R}_{i}^{*}$ and $\hat{\boldsymbol{U}}_{i}^{*}$.
Observe first that $\hat{R}_{i}=\hat{R}_{i}^{*}$ so that $s_{k,0}(\hat{R}_{i})=s_{k,0}(\hat{R}_{i}^{*})$
and component $\mathcal{R}_{K}$, whose elements span the $\Pi_{k,0}$
of Theorem \ref{thm:3.1}, is trivially $AL(m)$-invariant. For component
$\mathcal{U}_{K}$, whose elements span the $\mathcal{O}(m)$-invariant
$\Pi_{k,[k/2]}$, notice that $\pi_{k,0,\ell}(\hat{R}_{i}^{*},\hat{\boldsymbol{U}}_{i}^{*})=\pi_{k,0,\ell}(\hat{R}_{i},\mathbf{O}^{*}\hat{\boldsymbol{U}}_{i})$,
where \textcolor{black}{${\mathbf{O}^{*}=(\mathbf{A}\hat{\mathbf{V}}^{-1}\mathbf{A}^{T})^{-1/2}\mathbf{A}\hat{\mathbf{V}}^{-1/2}}$} $\in\mathcal{O}(m)$. 
It follows that $$(\pi_{k,0,\ell}(\hat{R}_{i},\mathbf{O}^{*}\hat{\boldsymbol{U}}_{i}),\ell=1,...,e_{m}([k/2]))^{T}=\mathbf{O}^{**}(\pi_{k,0,\ell}(\hat{R}_{i},\hat{\boldsymbol{U}}_{i}),\ell=1,...,e_{m}([k/2]))^{T},$$
for some $\mathbf{O}^{**}\in\mathcal{O}(e_{m}([k/2])$ by standard
properties of Wigner $d$-matrices. Hence each $(\bar{\pi}_{k,0,\ell},\ell=1,...,e_{m}([k/2]))$,
and thus $\mathcal{U}_{K}$, is invariant to $AL(m)$-transformations.
A similar argument applies to the parts $\left\Vert \boldsymbol{\bar{\pi}}_{\mathcal{I},j}\right\Vert ^{2},j=1,2,3$
of component $\mathcal{I}_{K}$ which span the intermediate $\mathcal{O}(m)$-invariant $\Pi_{k,k-2j}$. Finally, consider the term
$\textrm{tr}(\mathbf{c}_{a}\mathbf{c}_{a}^{T}\bar{\mathfrak{I}}_{a}\bar{\mathfrak{I}}_{a}^{T}),a=1,2$.
Because $\bar{\mathfrak{I}}_{a}$ averaged over all $(\hat{R}_{i}^{*},\hat{\boldsymbol{U}}_{i}^{*})$
= $\bar{\mathfrak{I}}_{a}\mathbf{O}_{a}^{**},$ it follows that $\mathcal{I}_{K}$
is $AL(m)$-invariant. Then so is global statistic $\mathcal{Q}_{K}$
as $\mathcal{Q}_{K}(\gamma(\boldsymbol{X}),\gamma^{*}(\hat{\eta}))=\mathcal{Q}_{K}(\boldsymbol{X},\hat{\eta})$.

To be diagnostic, the components in (\ref{eq:alternate expression for the test statistic})
must be further processed and, in particular, \cite{Henze1997} shows that
they must be scaled. This operation must be done with some care in
order for the scaled statistics to retain both $AL(m)$-invariance
and their meaningful interpretations. One possibility is the following;
let $\hat{\boldsymbol{\Sigma}}_{\mathcal{UU}}$ be a block diagonal
matrix where each block is the empirical covariance matrix of the
$\pi_{k,0,\ell}(\hat{R}_{i},\hat{\boldsymbol{U}}_{i})\in\Pi_{k,[k/2]}$
appearing in $\boldsymbol{\pi}_{\mathcal{U}}$. Define similarly $\hat{\boldsymbol{\Sigma}}_{\mathcal{II}},\text{\ensuremath{\hat{\boldsymbol{\Sigma}}_{\mathcal{RR}}}}$
for $\boldsymbol{\pi}_{\mathcal{I}},\boldsymbol{\pi}_{\mathcal{R}}.$
The scaled components $\mathcal{U}_{K}^{(s)}=n\boldsymbol{\bar{\pi}}_{\mathcal{U}}^{T}\hat{\boldsymbol{\Sigma}}_{\mathcal{UU}}^{-1}\boldsymbol{\bar{\pi}}_{\mathcal{U}}$,
$\mathcal{I}_{K}^{(s)}=n\boldsymbol{\bar{\pi}}_{\mathcal{I}}^{T}\hat{\boldsymbol{\Sigma}}_{\mathcal{II}}^{-1}\boldsymbol{\bar{\pi}}_{\mathcal{I}}$
and $\mathcal{R}_{K}^{(s)}=n\boldsymbol{\bar{\pi}}_{\mathcal{R}}^{T}\hat{\boldsymbol{\Sigma}}_{\mathcal{RR}}^{-1}\boldsymbol{\bar{\pi}}_{\mathcal{R}}$ 
are $AL(m)$-invariant and diagnostic. See Section \ref{sec:sec 6 Usefulness of Dx}
for some examples of their usefulness. 

\begin{rem}
\label{rem:Show the k-min}Statistic $\mathcal{Q}_{1}$
is a function of $\pi_{1,0,\ell}(r,u)\propto ru_{\ell}$ and is thus
basically a distance between $\bar{\boldsymbol{X}}$
and the mle $\hat{\boldsymbol{\mu}}$. As such, it can be useful in
discriminating $H_{0}$ but provides little Dx information in any
senses associated with the $\mathcal{Q}=\mathcal{UIR}$ decomposition. In the above, this
term is bundled into component $\mathcal{I}_{K}$. Similarly $\mathcal{Q}_{2}-\mathcal{Q}_{1}$,
which is spread over $\mathcal{I}_{K}$ and $\mathcal{R}_{K}$, is
related to $\textrm{tr}(\hat{\boldsymbol{V}}^{-1}\mathbf{S})$ ($\mathbf{S}$
is the empirical covariance matrix of the $\boldsymbol{X}_{i})$ and
again can be useful in discriminating $H_{0}$ but otherwise provides
little Dx insights. \cite{Csorgo1989} advises that one should use a powerful
test for the null hypothesis, followed by less formal procedures when
the null is rejected. Thus it can be a reasonable strategy to consider only the elements $\mathcal{Q}_{k}-\mathcal{Q}_{2}$
at the Dx stage.\end{rem}

\section{Applications of the smooth test methodology to some EC distributions}\label{sec:Section 5 applications}

The above smooth GoF test has the
desirable features listed in Section \ref{sec:Introduction}:
with $K$ chosen appropriately, it can
be made adaptively omnibus; it possesses Dx capabilities; and 
it is $AL(m)$-invariant. In addition, its behavior under both $H_{0}$
and $H_{1}$ can be conveniently approximated. In this section, we
develop the test strategy for two popular EC distributions. Two others are
treated in Appendix \ref{sec:Appeb-C}.

\subsection{The multivariate normal (MVN) distribution}\label{subsec:Sous section sur la MVN}

Consider the important problem of testing
the null hypothesis $H_{0}$ that a sample $\boldsymbol{X}_{1},\ldots,\boldsymbol{X}_{n}$
arises from an $m$-dimensional MVN distribution with density 
\[
f(\boldsymbol{x};\eta)=(2\pi)^{-m/2}\det(\mathbf{V}^{-1})^{1/2}\exp\left\{ -(\boldsymbol{x}-\boldsymbol{\mathbf{\mu}})^{T}\mathbf{V}^{-1}(\boldsymbol{x}-\boldsymbol{\mu})/2\right\} .
\]
Under $H_{0}$, $R^{2}$ has a $\chi_{m}^{2}$ distribution with $j-$th
moment $2^{j}\Gamma(m/2+j)/\Gamma(m/2)$. The elements of an orthonormal
basis satisfying (\ref{eq:3.4}) are, via (\ref{eq:3.5}),

\begin{equation}
s_{j,i}(r)=(-1)^{j}\sqrt{\frac{j!\:\Gamma(m/2)}{2^{i}\:\Gamma(m/2+j+i)}}\,\,L_{j}^{m/2+i-1}(r^{2}/2),\label{eq:4.1}
\end{equation}
where $L_{j}^{\alpha}(\cdot)$ is the \emph{j}\textendash th generalized
Laguerre polynomial of order $\alpha$; see Appendix~\ref{sec:Appeb-B} for
MATHEMATICA commands to generate these quantities along with those
for $\{\Psi_{k,j}(\cdot),j=1,\ldots,e_{m}(k)\}$. Tables \ref{tab:Table des polynomes pour m =00003D 2}
and \ref{tab:Table des polynomes pour m =00003D 3} list all $\pi_{k,j,\ell}(r,\boldsymbol{u})$
for $k=3,4,5$ and $m=2,3$.

Next, it is easy to see that for the MVN, $\zeta(\cdot)=1$ (defined
in Section \ref{subsec:8.2 Derivation-of-the decomposition}) and
$\mathbf{c}_{0},\mathbf{c}_{1},\mathbf{c}_{2}$ vanish so that the
$\mathcal{Q}=\mathcal{UIR}$ decomposition (\ref{eq: U I R decomp})
takes a particularly simple form. To get its explicit expression,
first compute $\hat{\mathbf{Y}}_{i}$, $\hat{R}_{i}$ and $\hat{\boldsymbol{U}}_{i}$
with $\hat{\boldsymbol{\mu}}$ = $\bar{\boldsymbol{X}}$ and $\hat{\mathbf{V}}^{-1}$
= $\mathbf{S}^{-1}$, where $(\bar{\boldsymbol{X}},\mathbf{S)}$ is
the mle of \emph{$(\boldsymbol{\mu},\mathbf{V})$}. Then all $\bar{\pi}_{k,j,\ell}=0$
when $k=1,$ 2 and in view of Remark \ref{rem: remark choix de lamnda},
this leads to $\mathcal{Q}_{1} = \mathcal{Q}_{2}=0$. Upon writing $\mathcal{C}_{k,j}^{2}=n\sum_{\ell=1}^{e_{m}(k-2j)}(\bar{\pi}_{k,j,\ell})^{2}$,
(\ref{eq:2.7}) becomes: 
\begin{equation}
\mathcal{Q}_{K}=\sum_{k=3}^{K}\sum_{j=0}^{[k/2]}\mathcal{C}_{k,j}^{2},\label{eq:4.3}
\end{equation}
which is asymptotically  $\chi_{v_{\mathcal{Q}}}^{2}$
under $H_{0}$, where $\nu_{\mathcal{Q}}=\sum_{k=3}^{K}C_{k}^{m+k-1}$.
Also
\begin{equation}
\mathcal{R}_{K}=\sum_{\substack{k=3\\k:\textrm{even}}}^{K}\mathcal{C}_{k,k/2}^{2}\label{eq:Expression de R_k}
\end{equation}
is, under $H_{0}$, asymptotically $\chi_{\nu_{\mathcal{R}}}^{2}$
where $\nu_{\mathcal{R}}=[(K-2)/2]$. Moreover,\begin{equation}
\mathcal{U}_{K}=\sum_{k=3}^{K}\mathcal{C}_{k,0}^{2}\label{eq:Exression de R3}
\end{equation}
has for reference distribution a
$\chi_{\nu_{\mathcal{U}}}^{2}$, where $\nu_{\mathcal{U}}=\sum_{k=3}^{K}e_{m}(k)$.
Finally, 
\begin{equation}
\mathcal{I}_{K}=\sum_{k=3}^{K}\sum_{j=1}^{[(k-1)/2]}\mathcal{C}_{k,j}^{2}.\label{eq:Expression de I_k}
\end{equation}
is approximately $\chi_{\nu_{\mathcal{I}}}^{2}$ with $\nu_{\mathcal{I}}=\nu_{\mathcal{Q}}-\nu_{\mathcal{U}}-\nu_{\mathcal{R}}$. 

To scale the components in the $\mathcal{Q}=\mathcal{UIR}$
decomposition, let $\hat{\mathbf{\boldsymbol{\Sigma}}}_{k,j}$ denote the empirical
covariance matrix of the $n$ random vectors $(\pi_{k,j,\ell}(\hat{R}_{i},\hat{\boldsymbol{U}}_{i}),\ell=1,\ldots e_{m}(k-2j))$
with empirical mean $\bar{\boldsymbol{\pi}}_{k,j}=(\bar{\pi}_{k,j,\ell},\ell=1,\ldots e_{m}(k-2j))$.
The scaled components are $(\mathcal{C}_{k,j}^{(s)})^{2}=n\bar{\boldsymbol{\pi}}_{k,j}^{T}\hat{\mathbf{\boldsymbol{\Sigma}}}_{k,j}^{-1}\bar{\boldsymbol{\pi}}_{k,j}$.
These $(\mathcal{C}_{k,j}^{(s)})^{2}$, bundled into $\mathcal{U}^{(s)},\mathcal{I}^{(s)}$
or $\mathcal{R}^{(s)}$ are diagnostic to identify the aspects of
the MVN not supported by the data. \textcolor{black}{The R package ECGoFTestDx computing
the above smooth test for the MVN has been deposited on 
CRAN}. An application to a  data set is shown in Appendix \ref{sec:Appeb-D}.

When \emph{m} = 1, $\Pi_{k}$ is spanned by $\pi_{k,[k/2],1}(r,u)$,
with $u=\pm1$. These are the Hermite polynomials of \cite{Rayner2009}'s smooth test of univariate normality. Thus our
approach generalizes their test.

For univariate distributions, the components of smooth test statistics
are often related to GoF tests that have been introduced from other
principles. The same occurs here and some components of $\mathcal{Q}_{K}$
turn out to be well-known GoF test statistics for the MVN. In particular,
\begin{flalign}
\mathcal{Q}_{4} & =\mathcal{C}_{3,0}^{2}+\mathcal{C}_{1,1}^{2}\qquad(=\mathcal{Q}_{3})\nonumber \\
 & +\mathcal{C}_{4,0}^{2}+\mathcal{C}_{4,1}^{2}+n(\bar{\pi}_{4,2,1})^{2}.\label{eq:4.5}
\end{flalign}
Inspection of these terms 
shows that  $\mathcal{R}_{4}=n(\bar{\pi}_{4,2,1})^{2}$ =
is $nb_{1,m}/6$ where $b_{1,m}$ is \cite{Mardia1970}'s 
multivariate measure of skewness. Moreover, $n(\bar{\pi}_{4,2,1})^{2}$
= $n(b_{2,m}-m(m+2))^{2}/(8m(m+2))$, where $b_{2,m}$ is \cite{Mardia1970}'s multivariate measure of kurtosis. These are popular measures
in multivariate analysis, whose usefulness in providing Dx information
arises by extending the meaning of their univariate counterparts:
in particular, $b_{1,m}$ vanishes under elliptical symmetry and $b_{2,m}$
quantifies the tails of the distribution. Consequently, they have
been used as GoF tests for the MVN and, in particular, \cite{Romeu1993} have concluded that tests based on them rank among the bests
in terms of power. \cite{Koizumi2009} have combined them (i.e.
$\mathcal{Q}_{3}+n(\bar{\pi}_{4,2,1})^{2}$) to get a multivariate
version of the popular Jarque-Bera test of univariate normality. 

\cite{Mardia1991} have developed a test of multivariate normality
as an example of some simplifications that occur with Rao\textquoteright s
score method when a suitable group structure exists. By brute force
calculations, they were able to obtain $\mathcal{Q}_{4}$ and state to have been unable to generalize to \emph{$K$}
$>$ 4. Here, by making use of recent advances in computational tools
for harmonic analysis, we can easily go beyond this limitation. Moreover,
as a by-product, we get for any \emph{$K$} the informative $\text{\ensuremath{\mathcal{Q}}=\ensuremath{\mathcal{UIR}}}$
decomposition. 

The idea of using the elements of stochastic representation (\ref{eq:3.1})
for testing MVN is not new. Many authors have proposed testing the
distribution of $R$ and/or $\boldsymbol{U}$ individually (see \cite{Koziol1986,Henze2002} for reviews), thus in effect crafting in an \emph{ad
hoc} fashion fragments of the $\text{\ensuremath{\mathcal{Q}}=\ensuremath{\mathcal{UIR}}}$
decomposition. \cite{Koziol1987}, followed by \cite{Rayner1988} and \cite{Best1988}, have taken a different route and proposed a smooth test based
on the fact that the multivariate normal density is the product of
\emph{m} univariate normal densities when $\mathbf{V}$ = $\mathbf{I}_{m}$.
For this particular case, a CONB can be obtained by the tensor product
of elements of CONB associated with each univariate distribution \cite[p.~51, Theorem~4.3]{Lancaster1969}.
 However, this approach breaks down for general
elliptical distributions, as in the following application, because
$\mathbf{V}$ = $\mathbf{I}$ is not associated with independence
\cite{Muirhead1982} and in view of the inconsistency
property of many EC distributions.

\subsection{The bivariate Laplace distribution}\label{subsec:Section bivariate Laplace}

There are several bivariate extensions of the univariate Laplace ;
we take here the variant discussed in \cite{Naik2006}
obtained from the power exponential distribution in Section \ref{sec:Introduction}
by setting $\alpha=1/2$. This distribution has larger tails than
the MVN. The $j-$th moment of $R^{2}$ is $\Gamma(2(j+1))$. We take
$K$= 5, a reasonable value, and from (\ref{eq:3.5}), the $s_{j,k-2j}(\cdot)$
required for $\mathcal{Q}_{5}$ are : $s_{0,1}(r)=(\sqrt{6})^{-1}$;
$s_{0,2}(r)=(2\sqrt{30})^{-1}$; $s_{0,3}(r)=(12\sqrt{35})^{-1}$;
$s_{0,4}(r)=(72\sqrt{70})$$^{-1}$; $s_{0,5}(r)=(720\sqrt{77})^{-1}$;
$s_{1,0}(r)=(r^{2}-6)/(2\sqrt{21})$; $s_{1,1}(r)=(r^{2}-20)/(4\sqrt{165})$;
$s_{1,2}(r)=(r^{2}-42)/(60\sqrt{42})$; $s_{1,3}(r)=(r^{2}-72)/(144\sqrt{665})$;
$s_{2,0}(r)=(7r^{4}-360r^{2}+1320)/(24\sqrt{10745})$ and $s_{2,1}(r)=(11r^{4}-10926r^{2}+12600)/(72\sqrt{225610})$.
Along with the expression for $\Psi_{k,\ell}(\cdot)$ in Section \ref{subsec:Sec3.1}
when $m=2$, it is an easy exercise to compute the $\pi_{k,j,\ell}(\cdot,\cdot)$.
Next, because $\zeta(r^{2})=r^{2}\times\sqrt{\frac{1}{r^{2}}}$, one
finds $\sigma_{1}=1$, $\sigma_{2}=6$ and
\begin{quote}
\begin{alignat*}{1}
\sqrt{d_{0}}\mathbf{c}_{0} & =(7)^{-1}\{-\sqrt{614},4\sqrt{\frac{10}{3}}\},\\
\sqrt{d_{1}}\mathbf{c}_{1} & =(\sqrt{1203})^{-1}\{\sqrt{4102},-4\sqrt{\frac{2051}{55}},48\sqrt{\frac{6}{55}}\},\\
\sqrt{d_{2}}\mathbf{c}_{2} & =(\sqrt{19})^{-1}\{-2\sqrt{35},4\}.
\end{alignat*}
\end{quote}
After computing the mle $\hat{\eta}=(\hat{\boldsymbol{\mu}},\hat{\mathbf{V}}^{-1})$
and in turn, $\hat{\boldsymbol{Y}}_{i}$ , $\hat{R}_{i}$ and $\hat{\boldsymbol{U}}_{i}$,
plug these into the $\pi_{k,j,\ell}(\hat{R}_{i},\hat{\boldsymbol{U}}_{i})$
to get the $\bar{\pi}_{k,j,\ell}$. Setting $\bar{\pi}_{\mathcal{U}}$ 
$=(\bar{\pi}_{3,0,1},\bar{\pi}_{3,0,2},\bar{\pi}_{4,0,1}$, $\bar{\pi}_{4,0,2},\bar{\pi}_{5,0,1},\bar{\pi}_{5,0,2})$,
$\bar{\pi}_{\mathcal{R}}=$ $(\bar{\pi}_{2,1,1},\bar{\pi}_{4,2,1}),$
$\bar{\pi}_{\mathcal{I},1}$ $=(\bar{\pi}_{1,0,1},\bar{\pi}_{3,1,1},\bar{\pi}_{5,1,1},\bar{\pi}_{5,2,1}$, $\pi_{1,0,2}, \bar{\pi}_{3,1,2})$,
$\bar{\pi}_{\mathcal{I},2}$ $=(\bar{\pi}_{2,0,1},\bar{\pi}_{4,1,1},\bar{\pi}_{2,0,2},\bar{\pi}_{4,1,2})$
and $\bar{\pi}_{\mathcal{I},3}=(\bar{\pi}_{5,1,2},\bar{\pi}_{5,2,2})$
yields all the elements required to compute the various test components.
Their degrees of freedom are those in Section \ref{sec: Section 4}.
For a GoF test for another bivariate Laplace distribution, see \cite{Fragiadakis2011}. 

\section{Performance of the $\mathcal{Q}=\mathcal{UIR}$ decomposition in
deriving Dx information}\label{sec:sec 6 Usefulness of Dx}

There can be many directions in which an alternative may depart from
a null model. The present Dx tool, based on the $\mathcal{Q}=\mathcal{UIR}$
decomposition, pertains to departures from the easily interpretable
representation (\ref{eq:3.1}). To evaluate
its performance, we need to determine how well it detects those departures,
ideally in a context where other departures coexist to complicate
matters. Performance here refers to the confidence allocated to the
Dx information extracted from a sample, i.e. sensibility and specificity.
As our tests are performed at a given level (here 5\%), specificity
is fixed and we concentrate on sensitivity, the probability that a
departure from representation (\ref{eq:3.1}) will be detected, which
is the power of the tests based on the components of the $\mathcal{Q}=\mathcal{UIR}$
decomposition.

To start building this confidence, a small experiment was performed.
We adopt the context of Section \ref{subsec:Sous section sur la MVN}
where the interest lies in assessing the null hypothesis of bivariate
normality. Because representation (\ref{eq:3.1}) is common to all
EC distributions, the results below could be representative of what
could be obtained with other EC null models and higher dimensions.
Samples were generated from thirty distributions taken
from \cite{Mecklin2005,Bogdan1999,Johnson1987}. Statistics $\mathcal{Q}_{K},\mathcal{U}_{K}^{(s)},\mathcal{I}_{K}^{(s)}$
and $\mathcal{R}_{K}^{(s)}$ , $(K=3,...,$12) were computed and compared
to their null Monte Carlo (based on 20,000 replications) approximation,
as explained in Remark \ref{rem:Remark 5 on Monte Carlo}. To
yield interesting power values,this was replicated 5,000 times and the sample size $n$ was adapted to
each alternative. 

We stress that our goal is to appreciate the usefulness of the Dx
procedure based of the $\mathcal{Q}=\mathcal{UIR}$ decomposition
and how these components relate to the global $\mathcal{Q}_{K}$.
It is \emph{not} to thoroughly compare the power of $\mathcal{Q}_{K}$
to its competitors, as we have made no attempt to optimize the choice
of $K$. Nevertheless, to offer some perspectives we did compute the
power of the BHEP test \cite{Baringhaus1988} %(Baringhaus and Henze, 1988)
with the
tuning parameter set at  1.41;  this test has been recommended in some simulation studies
\cite{Mecklin2005} but offers no Dx information.
We also computed the power of the multivariate Jarque \& Bera test
\cite{Koizumi2009} that does offer some Dx information via its
skewness-kurtosis components. 

The alternatives were chosen to somewhat resemble the bivariate normal
and to depart along one or two directions identifiable by the
$\mathcal{Q}=\mathcal{UIR}$ decomposition, in addition to other departures. To screen the alternatives
along these criteria, we generated 100 000 observations $(X_{1},X_{2})$
from each, transformed them into $(Y_{1},Y_{2})$ and then into $(R,\boldsymbol{U})$,
and then again $\boldsymbol{U}$ into $\theta=\arctan(U_{2}/U_{1})$.
Plots of the various univariate and bivariate densities were made, and various statistics and characteristics were
computed. We then excluded the densities where all $\mathcal{UIR}$
components are large, because the resulting knowledge (i.e. all is
wrong) is not easy to use in the iterative process of correcting a
null model after Dx analysis.

From the results from these thirty alternatives, we report here on
five that cover most of the behavior we have observed. Also, we report here
on the case $K=5$.

The first alternative is a member of the Khintchine family of distributions
\cite[Chap.~8]{Johnson1987}. Synthetic data from this distribution are
obtained by generating $Z\thicksim\Gamma(1.5,1)$ and then setting
\begin{alignat*}{1}
(X_{1},X_{2}) & =\sqrt{\frac{3\Gamma(1.5)}{\Gamma(1.5+2\times0.3998935)}}Z^{0.398935}\times2(V_{1}-0.5,V_{2}-0.5),
\end{alignat*}
where $V_{1},V_{2}$ are independent $U(0,1)$. In the above expression,
the various constants are such that the marginals of $(Y_{1},Y_{2})$
are nearly independent $N(0,1)$ and the marginal density of $R$
is approximately a $\sqrt{\chi_{2}^{2}}$ . However $\boldsymbol{U}$
is non-uniform, as the density of $\theta$
oscillates as $\sin(4\theta)$ with modes at $\pm\pi/4,\pm3\pi/4$. The density of $(Y_{1},Y_{2})$
has squarish contours but otherwise resembles the bivariate normal.
These oscillations induce a complicated relationship between $R$
and $\boldsymbol{U}$ which, at the first order, can be assimilated
to near independence, the regression of $R$ on $\theta$ involving
$\sin(k\theta)$, $\cos(k\theta)$, $k=1,...,3$ being nearly constant
with a Spearman's correlation coefficient ($\rho_{S}=-0.002)$ \footnote{Pearson's correlation coefficient is not a good measure of the dependency
between these quantities A normalized  version of mutual
information seems better
adapted but we have found its scale difficult to appreciate from
one problem to another. We have settled on Spearman's correlation
coefficient mainly because it remains the same between $\theta$ and
$R$ or $R^{2}$ and because a sophisticated user can in general make
some intuitive sense of differences between its values}. Because of these, we expect :
 $\textrm{power}(\mathcal{R}_{5}^{(s)})\approx\textrm{power}(\mathcal{I}_{5}^{(s)})\approx5\%<\textrm{power}(\mathcal{U}_{5}^{(s)}).$
Note that this, and further, expectation derives from the 100 000 samples
generated when screening the alternatives.
This was blinded in the following simulation study. The sample size
is $n=400.$ 

The second distribution is the generalized Burr-Pareto-Logistic (with
$\alpha=1,\beta=0)$ in \cite[Chap.~9]{Johnson1987}; see his Figure 9.8
for a sketch of the density and p. 167 for an algorithm to generate
synthetic data. The distribution of $(X_{1},X_{2})$ has dependent
$N(0,1)$ marginals with contours of triangular shape. Here the distribution
of $R$ is almost a $\sqrt{\chi_{2}^{2}}$ , $\boldsymbol{U}$ is
clearly non-uniform with a trimodal distribution, while $R$ and $\boldsymbol{U}$
are slightly correlated (with $\rho_{S}\approx0.04)$. Hence we should
find : $\textrm{power}(\mathcal{R}_{5}^{(s)})\approx5\%<\textrm{power}(\mathcal{I}_{5}^{(s)})\leq\textrm{power}(\mathcal{U}_{5}^{(s)}).$
The sample size is $n=250$.

The third distribution is the contaminated binormal
: $0.8\times MVN_{2}(0,\mathbf{I}_{2})+0.2\times MVN_{2}((1,1),\textrm{ Diag}\{1,2\})$,
see \cite[Chap.~4]{Johnson1987}. $R$ is again very
close to a $\sqrt{\chi_{2}^{2}}$ , $\boldsymbol{U}$ has a distribution
that slightly differs from uniformity while $R$ and $\boldsymbol{U}$
are slightly correlated ($\rho_{S}=-0.03$). Again
$n=400$ and we expect to find : $\textrm{power}(\mathcal{R}_{5}^{(s)})\approx5\%\leq\textrm{power}(\mathcal{U}_{5}^{(s)}),\textrm{power}(\mathcal{I}_{5}^{(s)})$,
with little insight about the comparative power of $\mathcal{U}_{5}^{(s)},\mathcal{I}_{5}^{(s)}.$

The fourth density is a Laplace-type (because of its marginals) distribution
generated by the following scheme : take $W_{0},W_{1},W_{2}\sim Exp(1)$
and form $(X_{1},X_{2})=(W_{1}-W_{0},W_{2}-W_{0})$. Here, $\boldsymbol{U}$
is non-uniform, $R$ departs from the $\sqrt{\chi_{2}^{2}}$ but $R$
and $\theta$ are approximately independent $(\rho_{S}=-0.004)$,
so one can expect: $\textrm{power}(\mathcal{I}_{5}^{(s)})\approx5\%\leq\textrm{power}(\mathcal{U}_{5}^{(s)}),\textrm{power}(\mathcal{R}_{5}^{(s)})$.
We have taken $n=75$ to mitigate the high power of $\mathcal{Q}_{5}$
with that of its components.

The last distribution is a $2$-th power
exponential distribution generated by taking $R\sim\Gamma(2,2)$,
$\theta\sim U(0,2\pi)$ and forming $(X_{1},X_{2})=R\times(\cos\theta,\sin\theta)$.
This yield an EC distribution where $R$ is markedly different from
a $\sqrt{\chi_{2}^{2}}$ . Thus here we should find : $\textrm{5\% \textrm{\ensuremath{\approx}\ power}(\ensuremath{\mathcal{U}_{5}^{(s)}}) \textrm{\ensuremath{\approx}\ power}(\ensuremath{\mathcal{I}_{5}^{(s)}}) \ensuremath{\leq}\ power}(\mathcal{R}_{5}^{(s)}).$
We have taken $n=100$.

Table \ref{tab:Approximate-power} shows the power (in \%) of the
various tests. The shaded cells are located
where the power of the components of the $\mathcal{Q}=\mathcal{UIR}$
decomposition are expected to be greater than nominal ($5\%$). 
A first general conclusion is that no global test dominates. An interesting comparison is between
J-B and $\mathcal{Q}_{5}$ where, as stated in Section \ref{subsec:Sous section sur la MVN},
the first is a fragment of the second. For some alternatives (e.g.
Burr-Pareto-Logistic), the extra components have little power and
lead to power dilution of $\mathcal{Q}_{5}$ with respect to J-B.
For other alternatives (e.g. Khintchine), power of J-B test is much
lower because the extra components in $\mathcal{Q}_{5}$ detect departures
that hardly translate into skewness or kurtosis. A data-driven approach
to select the value of $K$ could be one way to further increase the
power of $\mathcal{Q}_{K}$, but see Section \ref{sec:Conclusions Sec 7}. One atypical case
is the $2$-power exponential.
This will be discussed below. 

\begin{table}
\begin{centering}
\begin{tabular}{|c|c|c||c|c|c||c|c|c|c|}
\hline 
Alt. / Test & $n$ & BHEP & J-B & $\beta_{1}$ $=\mathcal{Q}_{3}$ & $\beta_{2}$ & $\mathcal{Q}_{5}$ & $\mathcal{U}_{5}^{(s)}$ & $\mathcal{I}_{5}^{(s)}$ & $\mathcal{R}_{5}^{(s)}$\tabularnewline
\hline 
\hline 
K  & 400 & 47.9 & 6.9 & 6.8 & 6.23 & 63.5 & \cellcolor[gray]{.85}97.6 & 4.6 & 5.4\tabularnewline
\hline 
BPL & 250 & 52.2 & 95.6 & 96.7 & 21.5 & 68.9 & \cellcolor[gray]{.85}94.9 & \cellcolor[gray]{.85}21.8 & 1.7\tabularnewline
\hline 
CMVN & 400 & 50.6 & 76.3 & 77.3 & 40.8 & 54.5 & \cellcolor[gray]{.85}19.5 & \cellcolor[gray]{.85}41.9 & 4.6\tabularnewline
\hline 
Lt & 75 & 96.6 & 99.0 & 98.1 & 95.0 & 94.9 & \cellcolor[gray]{.85}39.5 & 7.6 & \cellcolor[gray]{.85}21.9\tabularnewline
\hline 
2PE & 100 & 45.0 & 2.1 & 0.0 & 0.0 & 0.0 & 6.5 & 6.4 & \cellcolor[gray]{.85}92.2\tabularnewline
\hline 
\end{tabular}
\par\end{centering}
\caption{\label{tab:Approximate-power}Power in \% (based on 5,000 replications
at level $\alpha=5\%$) of the test of bivariate normality based on
$\mathcal{Q}_{5}$ and its $\mathcal{UIR}$ components. The reference
distribution has been approximated by 20,000 Monte Carlo samples of
size $n$ from the $MVN(\mathbf{0},\mathbf{I}_{2})$. Alternatives are K=Khintchine, BPL=Burr-Pareto-Logistic, CMVN=Contaminated MVN, Lt=Laplace -type and 2PE=$\alpha=2-$power exponential. Also shown is
the BHEP (Baringhauss-Henze-Epps-Pulley) test and the J-B (multivariate
Jarque-Bera) test along with its components $\beta_{1},\beta_{2}$
(both using Monte Carlo quantiles). The shaded cells correspond to
those where we expect power $>$ 5\%.}
\end{table}

Regarding the power (i.e. sensitivity) of the components, the Dx information
derived from the $\mathcal{Q}=\mathcal{UIR}$
decomposition is consonant with the expectations related to each
alternative. Furthermore, the magnitudes of the power in the shaded
cells somewhat reflect the severeness of the departures (when grossly
quantifiable). For
the Khintchine, both $\mathcal{R}_{5}^{(s)},\mathcal{I}_{5}^{(s)}$
have power approximately equal to level and the severe non-uniformity
of $\boldsymbol{U}$ is detected with good power. For the Burr-Pareto-Logistic,
the expectation : $\textrm{power}(\mathcal{R}_{5}^{(s)})\approx5\%<\textrm{power}(\mathcal{I}_{5}^{(s)})\leq\textrm{power}(\mathcal{U}_{5}^{(s)})$
is realized while the small power of $\mathcal{I}_{5}^{(s)}$ (21.8\%)
reflects the slight correlation between $R$ and $\boldsymbol{U}$.
Again, the severe non-uniformity of $\boldsymbol{U}$ is almost always
detected. For the contaminated MVN, $\mathcal{R}_{5}^{(s)}$ has,
as expected, trivial power, $\mathcal{I}_{5}^{(s)}$ picks up rather
well the dependency between $R$ and $\boldsymbol{U}$ while the slight
non-uniformity of $\boldsymbol{U}$ is detected with moderate power.
 For the Laplace-type case, the departures on the marginal distributions
of $R$ and $\boldsymbol{U}$ are correctly detected. For the $2-$power
exponential, $\textrm{power}(\mathcal{R}_{5}^{(s)})\approx91.6$,
showing that the departures along the distribution of $R$ are correctly
captured. The two other components have power close to nominal, as
they should. 

To summarize,  one can have some confidence that our Dx procedure
 can be a useful tool in the iterative process of modeling a data set.

In some cases, at least one component yield better power than the
global test statistic from which it is extracted. Thus the natural
two-stage strategy where one first test the global null hypothesis
using $\mathcal{Q}_{K}$ and, if significant, proceed at the Dx level,
would lose valuable information. This can be explained by a number
of factors. First, when only one component ought to be significant,
as with the Khintchine, the degrees of freedom associated with this
component (6 in this case) is much smaller than that of the global
test (here 15), thus diluting power. This effect is further confounded
by the scaling of the components, which also affects power, and suggests
looking at a scaled global statistic. Here $\textrm{power}(\mathcal{Q}_{5}^{(s)}=\mathcal{U}_{5}^{(s)}+\mathcal{I}_{5}^{(s)}+\mathcal{R}_{5}^{(s)})$
= 83.5, greater than the other global tests, which suggest replacing
$\mathcal{Q}_{5}$ by its scaled version. However the benefit of this
replacement is not universal : for the contaminated normal,
$\textrm{power}(\mathcal{Q}_{5})$ = 54.5 while $\textrm{power}(\mathcal{Q}_{5}^{(s)})=$
45.6. Thus it is not clear which of the above factors (degrees of
freedom vs scaling) has more impact on power. An extreme case of this
is the $2-$th power exponential where the power of $\mathcal{Q}_{5}$
is almost zero and the J-B test is not much better. Here scaling has
a large effect on power as $\textrm{power}(\mathcal{Q}_{5}^{(s)})$
= 47.5. This is an instance where the values of the $\lambda_{k,j,\ell}$
in Remark \ref{rem:Remark on power} renders the global test biased.
Scaling partly corrects the problem, as explained in \cite{Ducharme2016}. A procedure that would select in a data-driven
fashion between $\mathcal{Q}_{K}$ and $\mathcal{Q}_{K}^{(s)}$, in
addition to a proper value of $K$ globally or at the component level,
would perhaps reduce such variations. See Section \ref{sec:Conclusions Sec 7}.

The J-B test also offers some Dx information via its $b_{1,m},b_{2,m}$
components. But they are not scaled and, in particular, $\mathcal{R}_{5}^{(s)}$ is the scaled version of $b_{2,m}$. Comparing their powers shows the pitfalls in attempting to interpret unscaled Dx statistics. However, being fragments of our components, they can be rescaled
and exploited to further refine the Dx information, a problem we now look into from another angle.

In the above, we have taken $K=5$ and shown that our procedure can provide
useful Dx information supplementing the global test. Now we go further
and vary $K$ to see if more precise Dx information can be extracted
from a set of components.
Table \ref{tab:Values-of-some Table 2} presents further results from
our experiment, namely the values of $(\mathcal{U}_{k}^{(s)},k=3,...,7)$
and $(\mathcal{I}_{k}^{(s)},k=3,...,7)$ in the case of the Khintchine
distribution ($n=400)$. Notice the sharp increase in power
between $\mathcal{U}_{3}^{(s)}$ and $\mathcal{U}_{4}^{(s)}$. This
indicates that the hypersphericals in the added components of $\mathcal{U}_{4}^{(s)}$
pick up some important departure from uniformity. To see what as been
detected, we refer to the definition of projection-based Dx information
in \cite{Ducharme2016}. Recall that the elements in $\mathcal{U}_{3}^{(s)}$
have the form $\sin(k\theta)$, $\cos(k\theta)$, $k=1,2,3$ and
thus have at most 3 maximums in $[0,2\pi)$. Those
added in $\mathcal{U}_{4}^{(s)}$ have the form $\sin(4\theta)$,
$\cos(4\theta)$ and the increased power indicates that they capture
higher oscillations, suggesting that the density of $\boldsymbol{U}$
(or $\theta)$ could have at least four maximums. This is consonant with 
the fact that $\theta$ 
has a distribution that oscillates as $\sin(4\theta)$. Also, recall that the scatter plot of $(R,\theta)$ reveals
that these quantities have a complicated relationship. A similar increase of power from $\mathcal{I}_{5}^{(s)}$
to $\mathcal{I}_{6}^{(s)}$ and afterwards indicates again that some
correlation has been detected at the order of functions of the form
$\textrm{cos}(6\theta),\textrm{sin}(6\theta)$. Indeed a regression
of $R$ on $\theta$ with basis functions $\sin(k\theta),\cos(k\theta),k=1,...,7$
retains the functions $\cos(4\theta)$ and $\sin(6\theta)$ as significant.
Hence the results in Table \ref{tab:Values-of-some Table 2}
 indicate that beyond the basic analysis of the $\mathcal{Q=UIR}$
components, a finer analysis of the subcomponents could be exploited
to pinpoint with better precision the nature of the departures. 

\begin{table}
\begin{centering}
\begin{tabular}{|c|c|c|c|c|c|}
\hline 
Component & $K=3$ & $K=4$ & $K=5$ & $K=6$ & $K=7$\tabularnewline
\hline 
\hline 
$\mathcal{U}_{K}^{(s)}$ & 5.2 & 98.7 & 97.6 & 96.3 & 94.5\tabularnewline
\hline 
$\mathcal{I}_{K}^{(s)}$ & 4.8 & 4.0 & 4.6 & 56.1 & 47.1\tabularnewline
\hline 
\end{tabular}
\par\end{centering}
\caption{\label{tab:Values-of-some Table 2}Power in \% (based on 5000 replications)
of some components in the $\mathcal{UIR}$ decomposition for the Khintchine
alternative. The level is $\alpha=5\%.$ The reference distribution
has been approximated by 20 000 Monte Carlo samples of size $n=400$
from the $MVN(\mathbf{0},\mathbf{I}_{2})$.}

\end{table}

\section{Conclusions}\label{sec:Conclusions Sec 7}

This paper develops a smooth test of goodness-of-fit
for EC distributions that, through the invariant $\mathcal{\mathcal{Q}=\mathcal{UIR}}$
decomposition, allows to extract illuminating Dx information. Some applications
are worked out, in particular the important case of the MVN, where
the usefulness of the $\mathcal{Q}=\mathcal{UIR}$ decomposition is
demonstrated. For most EC distributions, there exists no competitor
to our approach. For the MVN, a small experiment suggests that our
smooth test can be competitive powerwise.

In the present work, no attempt is made to optimize
the selection of the hyperparameter $K$. his
topic deserves a separate analysis because the present
context opens up many new possibilities beyond the obvious extension
to select among $\mathcal{Q}_{1},\mathcal{Q}_{2},...,\mathcal{Q}_{K}$
in a data-driven fashion via (\ref{eq: Choix de K_hat}).  One of
these is to increase sensitivity of the Dx by selecting among the
components, say $\mathcal{U}_{1},\mathcal{U}_{2},...,\mathcal{U}_{K}$.
In addition, it was shown in Table \ref{tab:Values-of-some Table 2}
that this could also prove useful in better understanding the nature
of the detected departure. Another is to select among \{$\mathcal{U}_{K},\mathcal{U}_{K}+\mathcal{I}_{K},$
$\mathcal{U}_{K}+\mathcal{I}_{K}+\mathcal{R}_{K}=\mathcal{Q}_{K}\}$
or some other (data-driven) permutation of the components. Yet another
is to select between the raw components and their rescaled versions.

In multivariate contexts, the computation of the mle 
may be difficult or the information  $\boldsymbol{\mathcal{J}}_{\eta}$
may not exist. This is the case for example in the Laplace
discusses by \cite{Fragiadakis2011},
thus forcing the use of moment estimators. Deriving the $\mathcal{\mathcal{Q}=\mathcal{UIR}}$
decomposition in such cases remains to be investigated.

Finally, EC distributions often involve (e.g. the
$\alpha$-th power exponential), in addition
to $\eta$, a shape parameter that offers more flexibility in adjusting
the data. The shape parameter must in general be estimated and this
does not, in principle, affect the use of a test statistic similar
to (\ref{eq:2.6}). However, an open question is how does this estimation
impact the $\mathcal{Q}=\mathcal{UIR}$ decomposition. 

\bibliographystyle{imsart-nameyear}
\bibliography{biblio.bib}

%\section*{Acknowledgements}

%The authors would like to thank Bernard Boulerice for his contribution
%to a preliminary version of this paper. 

%\newpage{}
%

%

\newpage

\appendix

\section{Proofs and auxiliary results}\label{sec: 8 Proofs-and-auxiliary}

\subsection{Proof of Theorem \ref{thm:3.1}}\label{subsec:Proof-of-Theorem 3.1}

Let $k\geq0$ and $\boldsymbol{x}\mathbb{\in R}^{m}$. For any fixed
$\gamma$ = $(\mathbf{A},\boldsymbol{b})$ $\in$ $AL(m)$, define
the functions

\[
\tilde{\Psi}_{i,\ell}(\boldsymbol{x};\gamma)=\Psi_{i,\ell}(\gamma(\boldsymbol{x})/\Vert\gamma(\boldsymbol{x})\Vert)\Vert\gamma(\boldsymbol{x})\Vert^{i},
\]
where the $\Psi_{i,\ell}(\cdot)$ are the spherical harmonics evoked
in Section \ref{subsec:Sec3.1}. To avoid trivialities, set $\tilde{\Psi}_{i,\ell}(-\boldsymbol{b};\gamma)=0$
for $i>0$ and $\tilde{\Psi}_{0,1}(-\boldsymbol{b};\gamma)=1$. It
follows that $\tilde{\Psi}_{i,\ell}(x;\gamma)$ is an homogeneous
polynomial of degree \emph{$i$} in the components of $\gamma(\boldsymbol{x})=\mathbf{A}(\boldsymbol{x}+\boldsymbol{b})$
and thus a polynomial of degree \emph{$i$} in $\boldsymbol{x}\in\mathbb{R}^{m}$
. We need the following lemma which states that for any given $\gamma$,
any polynomial in $\mathbb{R}^{m}$ can be expressed as a linear combination
of the $\pi_{k,j,\ell}(r,\boldsymbol{u})$ of (\ref{eq:Forme explicite des polyn=0000F4mes Pi_k,j,l})
and explains how the coefficients of this linear combination are affected
by the choice of $\gamma$.
\begin{lem*}
Let $p(\cdot)\in\mathcal{P}_{k}$. Then, for any $\gamma=(\mathbf{A},\boldsymbol{b})\in AL(m)$,
we have 

\begin{equation}
p(\boldsymbol{x})=\sum_{i=0}^{k}\sum_{j=0}^{[(k-i)/2]}\sum_{\ell=1}^{e_{m}(i)}a_{j,i,\ell}(\gamma)s_{j,i}(\Vert\gamma^{-1}(\boldsymbol{x})\Vert)\tilde{\Psi}_{i,\ell}(\boldsymbol{x};\gamma^{-1}),\label{eq:A.1}
\end{equation}
for all $\boldsymbol{x}$ $\in$ $\mathbb{R}^{m}$, where $\gamma^{-1}=(\mathbf{A}^{-1},-\mathbf{A}\boldsymbol{b})$
and 

\begin{equation}
a_{j,i,\ell}(\gamma)=c_{m}\det(\mathbf{A})\intop_{\mathbb{R}^{m}}p(\boldsymbol{x})s_{j,i}(\Vert\gamma^{-1}(\boldsymbol{x})\Vert)\tilde{\Psi}_{i,\ell}(\boldsymbol{x};\gamma^{-1})\phi_{m}(\Vert\gamma^{-1}(\boldsymbol{x})\Vert^{2})d\boldsymbol{x}.\label{eq:A2}
\end{equation}
\end{lem*}
\begin{proof}
Let $\boldsymbol{x}=r\boldsymbol{u}$ where \emph{r} = $\Vert\boldsymbol{x}\Vert$
and \emph{$\boldsymbol{u}$} = $\boldsymbol{x}/\Vert\boldsymbol{x}\Vert$.
To avoid trivialities again, when \emph{$\boldsymbol{x}$} = 0, set
\emph{r} = 0, define \emph{$\boldsymbol{u}$} arbitrarily and agree
to take $r^{i}$ = 1. Consider the polynomial $p(\gamma(r\boldsymbol{u}))$
on $\mathbb{R}^{+}\times\Omega_{m}$. For fixed \emph{$r$}, this
is a polynomial of degree at most \emph{$k$} in the components of
$\boldsymbol{u}$. Now, because the spherical harmonics $\{\Psi_{i,\ell}(\cdot),\ell=1,\ldots,e_{m}(i),i\geq0\}$
form a CONB with respect to the scalar product $\left\langle \cdot,\cdot\right\rangle _{d\omega_{m}}$
for $L^{2}(\Omega_{m})$ , and thus for the restriction to $\Omega_{m}$
of the space of polynomials of degree \emph{$k$} on $\mathbb{R}^{m}$,
it follows that 
\[
p(\gamma(r\boldsymbol{u}))=\sum_{i=0}^{k}\sum_{\ell=1}^{e_{m}(i)}q_{i,\ell}(r;\gamma)\Psi_{i,\ell}(\boldsymbol{u}),
\]
where
\[
q_{i,\ell}(r;\gamma)=\frac{1}{\omega_{m}(\Omega_{m})}\intop_{\Omega_{m}}p(\gamma(r\boldsymbol{u}))\Psi_{i,\ell}(\boldsymbol{u})d\omega_{m}(\boldsymbol{u}).
\]
Setting $e=(\mathbf{I}_{m},\mathbf{0})$ as the identity element of
$AL(m)$, this shows that
\begin{flalign}
p(\gamma(\boldsymbol{x})) & =\sum_{i=0}^{k}\sum_{\ell=1}^{e_{m}(i)}w_{i,\ell}(\Vert\boldsymbol{x}\Vert;\gamma)\tilde{\Psi}_{i,\ell}(\boldsymbol{x};e)\nonumber \\
 & =\sum_{i=0}^{k}\sum_{\ell=1}^{e_{m}(i)}w_{i,\ell}(\Vert\boldsymbol{x}\Vert;\gamma)\tilde{\Psi}_{i,\ell}(\gamma(\boldsymbol{x});\gamma^{-1}),\label{eq:expression de p(gamma(x))}
\end{flalign}
where $w_{i,\ell}(\Vert\boldsymbol{x}\Vert;\gamma)$ = $\Vert\boldsymbol{x}\Vert^{-i}q_{i,\ell}(\Vert\boldsymbol{x}\Vert;\gamma)$.
Thus
\begin{equation}
p(\boldsymbol{x})=\sum_{i=0}^{k}\sum_{\ell=1}^{e_{m}(i)}w_{i,\ell}(\Vert\gamma^{-1}(\boldsymbol{x})\Vert;\gamma)\tilde{\Psi}_{i,\ell}(\boldsymbol{x};\gamma^{-1}).\label{eq:A3}
\end{equation}
Now, because $\tilde{\Psi}_{i,\ell}(\boldsymbol{x};\gamma^{-1})$
is a polynomial of degree \emph{$i$} in \emph{$\boldsymbol{x}$}
and $p(\gamma(\boldsymbol{x}))$ is a polynomial of degree at most
$k$, it follows that each $w_{i,\ell}(\Vert\boldsymbol{x}\Vert;\gamma)$
is a polynomial of degree at most $[(k-i)/2]$ in $\Vert\boldsymbol{x}\Vert^{2}$.
Hence in turn, it can be written as a linear combination of the polynomials
$s_{j,i}(\Vert\boldsymbol{x}\Vert)$, $i=0,\ldots,[(k-i)/2]$ of (\ref{eq:3.4}),
namely
\begin{equation}
w_{i,\ell}(r;\gamma)=\sum_{j=0}^{[(k-i)/2]}a_{j,i,\ell}(\gamma)s_{j,i}(r).\label{eq:A4}
\end{equation}
Combining (\ref{eq:A3}) and (\ref{eq:A4}) yields (\ref{eq:A.1}).
Now using (\ref{eq:3.2}) and (\ref{eq:3.4}) with (\ref{eq:A4}),
we have
\begin{flalign}
a_{j,i,\ell}(\gamma) & =\omega_{m}(\Omega_{m})c_{m}\intop_{0}^{\infty}w_{i,\ell}(r;\gamma)s_{j,i}(r)\phi_{m}(r^{2})r^{m-1+2j}dr,\\
 & =c_{m}\intop_{\mathbb{R}^{m}}w_{i,\ell}(\Vert\boldsymbol{x}\Vert;\gamma)s_{j,i}(\Vert\boldsymbol{x}\Vert)\phi_{m}(\Vert\boldsymbol{x}\Vert^{2})\Vert\boldsymbol{x}\Vert^{2j}d\boldsymbol{x},\\
 & ={\scriptsize c_{m}\det(\mathbf{A})\intop_{\mathbb{R}^{m}}w_{i,\ell}(\Vert\gamma^{-1}(\boldsymbol{x})\Vert;\gamma)s_{j,i}(\Vert\gamma^{-1}(\boldsymbol{x})\Vert)\phi_{m}(\Vert\gamma^{-1}(\boldsymbol{x})\Vert^{2})\Vert\gamma^{-1}(\boldsymbol{x})\Vert^{2i}d\boldsymbol{x}},\label{eq:A5}
\end{flalign}
upon transforming from polar to cartesian coordinates. Finally, substitute
(\ref{eq:A3}) into the right-hand side of (\ref{eq:A2}). Simple
algebra using the orthogonality of the $\Psi_{i,\ell}(\cdot)$ yields
the right-hand side of (\ref{eq:A5}). This concludes the proof of
the lemma.
\end{proof}
Proof of \emph{i}). A change of variable in \ref{eq:3.6} shows that,
without loss of generality, we can set $\eta=(\boldsymbol{0},\mathbf{I}_{m}).$
By (\ref{eq:3.6}), when $k\neq k^{\prime}$, the spaces $\Pi_{k}$
and $\Pi_{k^{\prime}}$ are orthogonal with respect to $\left\langle \cdot,\cdot\right\rangle {}_{(f,\eta)}$.
By the lemma, all polynomials $p(\cdot)$ in $\mathcal{P}_{k}$ have
an expansion of the form (\ref{eq:A.1}) where from (\ref{eq:A5})
we can take $\gamma$ = $e$ to simplify the derivations, while all
$p_{k^{\prime}}(\gamma(\cdot))\in\Pi_{k^{\prime}}$ have an expansion
of the form \ref{eq:expression de p(gamma(x))}. Substituting these
expansions into (\ref{eq:3.6}) shows, after careful identification
and using the orthonormality of the $s_{j,i}(\cdot),$ $\Psi_{i,\ell}(\cdot)$,
that $p(\cdot)$ $\in$ $\Pi_{k}$ if and only if its coefficient
in (\ref{eq:A.1}) satisfies $a_{j,i,\ell}(e)$ = 0 for every \emph{i},
\emph{j} such that \emph{i} $\neq k-2j$.\textcolor{red}{{} }Thus, after
rearranging the indices and extending to any $\gamma$ $\in$ $AL(m)$,
each $p(\cdot)$ $\in$ $\Pi_{k}$ has an expansion of the form
\begin{equation}
p(\boldsymbol{x})=\sum_{j=0}^{[k/2]}\sum_{\ell=1}^{e_{m}(j)}a_{j,k-2j,\ell}(\gamma)s_{j,k-2j}(\Vert\gamma^{-1}(\boldsymbol{x})\Vert)\tilde{\Psi}_{k-2j,\ell}(\boldsymbol{x};\gamma^{-1}).\label{eq:A6}
\end{equation}
Because the set of polynomials $s_{j,k-2j}(\Vert x\Vert)\tilde{\Psi}_{k-2j,\ell}(\boldsymbol{x})$
are orthonormal with respect to $\left\langle \cdot,\cdot\right\rangle {}_{(f,\eta)}$
and span $\Pi_{k}$, it follows from (\ref{eq:A.1}) and (\ref{eq:A6})
that $\mathcal{P}_{k}$ = $\Pi_{k}\oplus\mathcal{P}_{k-1}$. This
proves the first part of \emph{i}). The second part comes from the
fact that the polynomials are dense in $L_{AL(m)}^{2}(f,\eta)$.

Proof of \emph{ii}) By construction, $\Pi_{k}$ is an $AL(m)$-invariant
subspace. From (\ref{eq:A6}), for any \emph{$\gamma$} $\in$ $AL(m)$,
there is only one element of $\Pi_{k}$, up to a multiplicative constant,
that is invariant over the subgroup of rotations. It follows from
Schur\textquoteright s lemma that $\Pi_{k}$ is irreducible. Indeed,
Schur\textquoteright s lemma \cite[p.~390]{Helgason1984} asserts that
if two representations of a group are irreducible then every commuting
non-zero linear map between them is an isomorphism and that a representation
is irreducible if and only if every such linear map is a scalar multiple
of the identity element. In the present context, showing that the
space $\text{\ensuremath{\Pi_{k}}}$ is irreducible amounts to show
that the representation of $AL(m)$ induced by the coefficients $a_{k,j,\ell}(\gamma)$
of its basis in (\ref{eq:A6}) is irreducible. The identity element
is then the element that is invariant over rotations which are identified
with the commuting maps.

Setting $g^{-1}=(V^{-1/2},-V^{-1/2}\mu)$, the
$$\pi_{k,j,\ell}(r,u)= s_{j,k-2j}(\Vert g^{-1}(x)\Vert)\tilde{\Psi}_{k-2j,\ell}(x,g^{-1})$$
are orthonormal with respect to $\left\langle \cdot,\cdot\right\rangle {}_{(f,\eta)}$
and, from (\ref{eq:A6}), form a complete basis for $\Pi_{k}$. Let
$\Pi_{k,j}(g^{-1})$ be the space generated by the span of $\{s_{j,k-2j}(\Vert g^{-1}(x)\Vert\tilde{\Psi}_{k-2j,\ell}(x;g^{-1})$,
$\ell$= $1,\ldots,e_{m}(k-2j)\}$. Because, for fixed $(k,j)$, $\Pi_{k,j}(g^{-1})$
is closed under rotations, the functions that span $\Pi_{k,j}(g^{-1})$
are, up to multiplicative constants, the same as those that span $E_{m}(k-2j)$.
Thus $\Pi_{k,j}(g^{-1})$ is an irreducible rotation\textendash invariant
subspace of $\Pi_{k}$ of dimension $e_{m}(k-2j)$. Combining these
functions shows that for $g^{-1}$ = $(V^{-1/2},-\mu)$, $\Pi_{k}$
= $\Pi_{k,0}(g^{-1})\oplus\cdots\oplus\Pi_{k,[k/2]}(g^{-1})$, with
respect to $\left\langle \cdot,\cdot\right\rangle {}_{(f,\eta)}$.
The dimension of $\Pi_{k}$ is computed from this direct sum as a
combinatorial exercise.

\subsection{Derivation of the $\mathcal{Q}=\mathcal{UIR}$ decomposition}\label{subsec:8.2 Derivation-of-the decomposition}

We derive (\ref{eq:alternate expression for the test statistic})
and (\ref{eq: U I R decomp}).
First reparametrize $\boldsymbol{\eta}=(\boldsymbol{\mu},\mathbf{V}^{-1})$
into the $m+m(m+1)/2$ vector $\boldsymbol{\eta}=(\boldsymbol{\mu},\textrm{lvec}(\mathbf{V}^{-1}))^{T}$,
where $\textrm{lvec}(\cdot)$ is the ``lower vec'' operator. Introduce the $m^{2}\times m(m+1)/2$ matrix $\mathbf{B}$
with element :
\begin{alignat*}{1}
b_{ij,k\ell} & =\left\{ \begin{array}{cc}
1 & \textrm{if }(i=k\textrm{ and }j=\ell)\textrm{ or }(i=\ell\textrm{ and }j=k)\\
0 & \mathscr{\textrm{otherwise}}
\end{array},\right.
\end{alignat*}
where $i$ first goes from 1
to $m$ and then $j$ goes from 1 to $m$ while $k$ varies from $\ell$
to $m$ followed by $\ell$ varying from 1 to $m$. This matrix, which
has rank $m(m+1)/2$, is such that $\textrm{vec}(\mathbf{V}^{-1})=\mathbf{B}\cdot\textrm{lvec}(\mathbf{V}^{-1})$.
Now some algebra yields
\begin{align*}
\frac{\partial\log f(\boldsymbol{x};\boldsymbol{\eta})}{\partial\boldsymbol{\eta}}= & \mathbf{N}_{\eta}\boldsymbol{Z},
\end{align*}
where ($\otimes$ denotes Kronecker's product)
\begin{align*}
\mathbf{N}_{\eta} & =\left(\begin{array}{cc}
\mathbf{V}^{-1/2} & \mathbf{0}\\
\mathbf{0} & \frac{1}{2}\mathbf{B}^{T}(\mathbf{V}^{-1/2}\otimes\mathbf{V}^{-1/2})\mathbf{B}
\end{array}\right),
\end{align*}

\begin{align}
\boldsymbol{Z} & =\left(\begin{array}{c}
\boldsymbol{Z}_{1}\\
\boldsymbol{Z_{2}}
\end{array}\right)=\left(\begin{array}{c}
R\times g(R^{2})\times\boldsymbol{U}\\
\textrm{lvec}(\mathbf{I}_{m}-\zeta(R^{2})\boldsymbol{U}\boldsymbol{U}^{T})
\end{array}\right),\label{eq:expression de Z-1}
\end{align}
with $g(r^{2})=$$\left.-2\phi_{m}^{\prime}(s)/\phi_{m}(s)\right|_{s=r^{2}}$
and $\zeta(r^{2})=r^{2}\times g(r^{2})$. Note that $\mathbf{N}_{\eta}$
is invertible and the distribution of $\boldsymbol{Z}$ does not depend
on $\boldsymbol{\eta}$; it has expectation 0 and, in view of 
results about the $U(\Omega_{m})$ distribution (see \cite[p.~239, Ex.~4]{Bilodeau1999}%Bilodeau and Brenner, 1999, p. 239, Ex. 4)
, its variance is block-diagonal with
the inverse of these blocks being :
\begin{alignat}{1}
\mathbb{V}^{-1}\left(\mathbf{Z}_{1}\right) & =m\mathbf{I}_{m}/\sigma_{1},\nonumber \\
 \mathbb{V}^{-1}\left(\mathbf{Z}_{2}\right) & ={\scriptsize\sigma_{2}[\mathbf{I}_{m(m+1)/2}-\frac{1}{2}\textrm{Diag(}\textrm{lvec}(\mathbf{I}_{m}))]-\frac{\sigma_{2}(1-\sigma_{2})}{4+2m(1-\sigma_{2})}\textrm{lvec}(\mathbf{I}_{m})\textrm{lvec}^{T}(\mathbf{I}_{m})},\label{eq: blocs de l'info de Fisher-1}
\end{alignat}
where $\sigma_{1}=\mathbb{E}_{0}(R^{2}\times g^{2}(R^{2}))$ and $\sigma_{2}=m(m+2)/\mathbb{E}_{0}(\zeta^{2}(R^{2}))$.
Then $\boldsymbol{\mathcal{J}}_{\eta}=\mathbf{N}_{\eta}\mathbb{V}(\boldsymbol{Z})\mathbf{N}_{\eta}^{T}$
is also block-diagonal. Moreover, the lines of $\mathbf{J}_{\eta}$
are 
{\scriptsize\begin{alignat*}{1}
\mathcal{\mathbb{C}}\mathrm{o}\mathrm{v}_{0}\left(\pi_{k,j,\ell}(R,\boldsymbol{U}),\partial\log f(\boldsymbol{X};\eta)/\partial\boldsymbol{\eta}^{T}\right) & =\mathbf{N}_{\eta}\,\mathbb{E}_{0}\left(R^{k-2j}s_{j,k-2j}(R^{2})g(R^{2})\Psi_{k-2j,\ell}(\boldsymbol{U})\boldsymbol{Z}^{T}\right).
\end{alignat*}}
Hence the elements of $\text{\ensuremath{\mathbf{J}_{\eta}\boldsymbol{\mathcal{J}}_{\eta}^{-1}\mathbf{J}_{\eta}^{T}}}$
do not involve $\boldsymbol{\eta}$ and $(\mathbf{I}_{v_{\mathcal{Q}}}-\mathbf{J}_{\eta}\boldsymbol{\mathcal{J}}_{\eta}^{-1}\mathbf{J}_{\eta}^{T})$
depends only on the distribution of $R$. Defining $\boldsymbol{\Psi}_{k}=(\Psi_{k,\ell}(\boldsymbol{U}),\ell=1,\ldots,e_{m}(k))^{T}$,
it follows that $\boldsymbol{U}=\boldsymbol{\Psi}_{1}/\sqrt{m}$ while
$\textrm{lvec}(\boldsymbol{U}\boldsymbol{U}^{T})=\mathbf{A}_{m}(\boldsymbol{\Psi}_{0}^{T},\boldsymbol{\Psi}_{2}^{T})^{T}$,
where $\mathbf{A}_{m}$ is a square matrix of order $m(m+1)/2$) partitioned
as $(\mathbf{A}_{m(1)}\,\brokenvert\,\mathbf{A}_{m(2)})$ with $\mathbf{A}_{m(1)}=\textrm{lvec}(\mathbf{I}_{p})/m$
being its first column. The expression of $\mathbf{A}_{m(2)}$
is not required beyond the relationships $\textrm{lvec}(\mathbf{I}_{p})^{T}\mathbf{A}_{m(2)}=\mathbf{0}$
and $\mathbf{A}_{m(2)}^{T}\mathbf{A}_{m(2)}=\frac{1}{m(m+2)}\mathbf{I}_{\frac{m(m+1)}{2}-1}$,
from the expression for $\mathbb{V}(\textrm{lvec}(\boldsymbol{U}\boldsymbol{U}^{T}))$
in \cite[p. 239, Ex. 4]{Bilodeau1999}%Bilodeau and Brenner (1999, p. 239, Ex. 4)
. It follows from the
orthogonality of the $\Psi_{k-2j,\ell}(\cdot)$ that
{\scriptsize\begin{align}
\mathbb{E}_{0}\left(R^{k-2j}s_{j,k-2j}(R^{2})\boldsymbol{\Psi}_{k-2j}\boldsymbol{Z}_{1}^{T}\right) & =\begin{cases}
\mathbb{E}_{0}\left(s_{j,1}(R^{2})\zeta(R^{2})\right)\mathbf{I}_{m}/\sqrt{m} & \textrm{ if }k-2j=1\\
\mathbf{0} & \textrm{otherwise}
\end{cases},\label{eq: ligne non nulle des 2 premi=0000E8res colonnes de J_eta-1}
\end{align}}
while adding the orthogonality of the $s_{j,i}(\cdot)$,
{\scriptsize\begin{align}
\mathbb{E}_{0}\left(R^{k-2j}s_{j,k-2j}(R^{2})\boldsymbol{\Psi}_{k-2j}\boldsymbol{Z}_{2}^{T}\right) & =\begin{cases}
-\mathbb{E}_{0}\left(s_{j,0}(R^{2})\zeta(R^{2})\right)\mathbf{A}_{m(1)}^{T} & \textrm{ if }k-2j=0\\
-\mathbb{E}_{0}\left(R^{2}s_{j,2}(R^{2})\zeta(R^{2})\right)\mathbf{A}_{m(2)}^{T} & \textrm{ if }k-2j=2\\
\mathbf{0} & \textrm{otherwise }
\end{cases}.\label{eq:ligne non nulle des dernires col. de J_eta cas R-1}
\end{align}}
Collect the constants in these expressions
into the following vectors : for $k-2j=0$, $\mathbf{c}_{0}^{T}=\left(-\mathbb{E}_{0}\left(s_{j,0}(R^{2})\zeta(R^{2})\right),j=1,...,\left[K/2\right]\right)$,
for $k-2j=1,$ $\mathbf{c}_{1}^{T}=\left(\mathbb{E}_{0}\left(s_{j,1}(R^{2})\zeta(R^{2})\right)/\sqrt{m},j=1,...,\left[(K+1)/2\right]\right)$
and for $k-2j=2$, $\mathbf{c}_{2}^{T}=\left(-\mathbb{E}_{0}\left(R^{2}s_{j,2}(R^{2})\zeta(R^{2})\right),j=1,...,\left[K/2\right]\right)$.
Permute the components of $\left(\pi_{k,j,\ell}(\cdot,\cdot),k=1,...,K,(j,\ell)\in B_{k}\right)^{T}$
into $(\boldsymbol{\pi}_{\mathcal{U}}^{T},\boldsymbol{\pi}_{\mathcal{I}}^{T},\boldsymbol{\pi}_{\mathcal{R}}^{T})^{T}$
so that all $\pi_{k,0,\ell}(\cdot,\cdot)$ with $k\geq3$ are in $\boldsymbol{\pi}_{\mathcal{U}}$
in increasing order of $\ell$ and then $k$, all $\pi_{k,[k/2],1}(\cdot,\cdot)$
with $k\geq2$ are in $\boldsymbol{\pi}_{\mathcal{R}}$ while all other terms are regrouped in $\boldsymbol{\pi}_{\mathcal{I}}$
in the following way : let $\mathfrak{I}_{1}$ be the matrix with
rows $(\pi_{k,j,\ell}(r,\boldsymbol{u})\,\left|\,k-2j=1,\right.\ell=1,...,e_{m}(k-2j))$
ordered from row to row according to $k$. Similarly define $\mathfrak{I}_{2}$
with $k-2j=2$ and regroup all other terms (those with $k-2j>2$)
into $\boldsymbol{\pi}_{\mathcal{I},3}$ ordered lexicographically.
Then $\boldsymbol{\pi}_{\mathcal{I}}^{T}=(\textrm{vec}(\mathfrak{I}_{1})^{T},\textrm{vec}(\mathfrak{I}_{2})^{T},\boldsymbol{\pi}_{\mathcal{I},3}^{T})^{T}$$=(\boldsymbol{\pi}_{\mathcal{I},1}^{T},\boldsymbol{\pi}_{\mathcal{I},2}^{T},\boldsymbol{\pi}_{\mathcal{I},3}^{T})^{T}$.
Form $(\boldsymbol{\bar{\pi}}_{\mathcal{U}}^{T},\boldsymbol{\bar{\pi}}_{\mathcal{I}}^{T},\boldsymbol{\bar{\pi}}_{\mathcal{R}}^{T})^{T}$,
the permuted version of (\ref{eq:2.6}) ordered in the same way with
$\boldsymbol{\bar{\pi}}_{\mathcal{I}}^{T}=(\boldsymbol{\bar{\pi}}_{\mathcal{I},1}^{T}=\textrm{vec}(\bar{\mathfrak{I}}_{1})^{T},\boldsymbol{\bar{\pi}}_{\mathcal{I},2}^{T}=\textrm{vec}(\bar{\mathfrak{I}}_{2})^{T},\boldsymbol{\bar{\pi}}_{\mathcal{I},3}^{T})^{T}$.
Easy calculations show that $\mathbf{J}_{\eta}$ with
rows permuted accordingly has the form
\begin{equation}
\mathbf{J}_{\eta}=\mathbf{\mathbf{N}_{\eta}}\left(\begin{array}{cc}
\begin{array}{c}
\mathbf{0}\\
\nu_{\mathcal{U}}\times m
\end{array} & \begin{array}{c}
\mathbf{0}\\
\nu_{\mathcal{U}}\times m(m+1)/2
\end{array}\\
\begin{array}{c}
\boldsymbol{\tau}_{11}\\
\nu_{\mathcal{I}}\times m
\end{array} & \begin{array}{c}
\boldsymbol{\tau}_{12}\\
\nu_{\mathcal{I}}\times m(m+1)/2
\end{array}\\
\begin{array}{c}
\mathbf{0}\\
\nu_{\mathcal{R}}\times m
\end{array} & \begin{array}{c}
\boldsymbol{\tau}_{22}\\
\nu_{\mathcal{R}}\times m(m+1)/2
\end{array}
\end{array}\right),\label{eq: Matrix J_eta-1}
\end{equation}
where $\nu_{\mathcal{R}}=[K/2]$ , $\nu_{\mathcal{U}}=\sum_{k=1}^{K}e_{m}(k)$
, $\nu_{\mathcal{I}}=\nu_{\mathcal{Q}}-\nu_{R}-\nu_{\mathcal{U}}$
where $v_{\mathcal{Q}}=\sum_{k=1}^{K}C_{k}^{m+k-1}$ . Recall that
the $\boldsymbol{\tau}_{ij}$ depend only on the distribution of $R$
under $H_{0}$. Thus
\begin{equation}
(\mathbf{I}_{v_{\mathcal{Q}}}-\mathbf{J}_{\eta}\boldsymbol{\mathcal{J}}_{\eta}^{-1}\mathbf{J}_{\eta}^{T})=\left(\begin{array}{ccc}
\mathbf{I}_{\nu_{\mathcal{U}}} & 0 & 0\\
0 & \boldsymbol{\Sigma}_{\mathcal{II}} & \boldsymbol{\Sigma}_{\mathcal{IR}}\\
0 & \boldsymbol{\Sigma}_{\mathcal{IR}}^{T} & \boldsymbol{\Sigma}_{\mathcal{RR}}
\end{array}\right).\label{eq:mat var-cov cas elliptic general-1}
\end{equation}
We focus on $\boldsymbol{\Sigma}_{\mathcal{IR}}=-\boldsymbol{\tau}_{12}\mathbb{V}^{-1}(\boldsymbol{Z}_{2})\boldsymbol{\tau}_{22}^{T}$.
The lines of matrix $\boldsymbol{\tau}_{22}$ have the form $-\mathbb{E}_{0}\left(s_{j,0}(R^{2})\zeta(R^{2})\right)(\textrm{lvec}(\mathbf{I}_{p}))^{T}$.
Combining this with the expression of $\mathbb{V}^{-1}\left(\mathbf{Z}_{2}\right)$
and the fact that the lines of $\boldsymbol{\tau}_{12}$ either vanish
or, from (\ref{eq:ligne non nulle des dernires col. de J_eta cas R-1})
have the form $-\mathbb{E}_{0}\left(R^{2}s_{j,2}(R^{2})\zeta(R^{2})\right)\mathbf{A}_{m(2)}^{T}$,
it follows that $\boldsymbol{\Sigma}_{\mathcal{IR}}=0$. Also $\boldsymbol{\Sigma}_{\mathcal{II}}$
= $\mathbf{I}_{\nu_{\mathcal{I}}}-\boldsymbol{\tau}_{11}\mathbb{V}^{-1}(\boldsymbol{Z}_{1})\boldsymbol{\tau}_{11}^{T}-\boldsymbol{\tau}_{12}\mathbb{V}^{-1}(\boldsymbol{Z}_{2})\boldsymbol{\tau}_{12}^{T}$
and $\boldsymbol{\Sigma}_{\mathcal{RR}}$ = $\mathbf{I}_{\nu_{\mathcal{R}}}-\boldsymbol{\tau}_{22}\mathbb{V}^{-1}(\boldsymbol{Z}_{2})\boldsymbol{\tau}_{22}^{T}$.
It is easy to see that $\boldsymbol{\Sigma}_{\mathcal{RR}}^{-1}=I_{[K/2]}+d_{0}\mathbf{c}_{0}\mathbf{c}_{0}^{T}$,
where $d_{0}=\frac{\sigma_{2}}{m(2+m(1-\sigma_{2})-\sigma_{2}\left\Vert c_{0}\right\Vert ^{2}}$.
It is again easy to see that $\mathbf{\boldsymbol{\Sigma}}_{\mathcal{II}}^{-1}$
is block diagonal, with block $\mathbf{\boldsymbol{\Sigma}}_{\mathcal{II},1}^{-1}$
= $\mathbf{I}_{e_{m}(1)\times[(K+1)/2]}+\mathbf{I}_{e_{m}(1)}\varotimes d_{1}\mathbf{c}_{1}\mathbf{c}_{1}^{T}$
, where $d_{1}=\frac{m}{\sigma_{1}-m\left\Vert c_{1}\right\Vert ^{2}}$
, block $\mathbf{\boldsymbol{\Sigma}}_{\mathcal{II},2}^{-1}=\mathbf{I}_{e_{m}(2)\times[K/2]}+\mathbf{I}_{e_{m}(2)}\varotimes d_{2}\mathbf{c}_{2}\mathbf{c}_{2}^{T}$
with $d_{2}=\frac{\sigma_{2}}{m(2+m)-\sigma_{2}\left\Vert c_{2}\right\Vert ^{2}}$
and block $\mathbf{\boldsymbol{\Sigma}}_{\mathcal{II},3}^{-1}$ $=\mathbf{I}$$_{\nu_{\mathcal{I}}-e_{m}(1)([(K+1)/2]+[K/2])}$.
Using standard properties of the $\textrm{vec}$ and $\varotimes$
operators, we get $\boldsymbol{\pi}_{\mathcal{I},1}^{T}\mathbf{\boldsymbol{\Sigma}}_{\mathcal{II},1}^{-1}\boldsymbol{\pi}_{\mathcal{I},1}$$=d_{1}tr(\mathbf{c}_{1}\mathbf{c}_{1}^{T}\bar{\mathfrak{I}}_{1}\bar{\mathfrak{I}}_{1}^{T})$
and similarly for the other term. Collecting these, we finally get
:
\begin{align}
  \mathcal{Q}_{K} & = & n\left[\left\Vert \boldsymbol{\bar{\pi}}_{\mathcal{U}}\right\Vert ^{2}+\left\Vert \boldsymbol{\bar{\pi}}_{\mathcal{I},1}\right\Vert ^{2}+d_{1}\textrm{tr}(\mathbf{c}_{1}\mathbf{c}_{1}^{T}\bar{\mathfrak{I}}_{1}\bar{\mathfrak{I}}_{1}^{T})+\left\Vert \boldsymbol{\bar{\pi}}_{\mathcal{I},2}\right\Vert ^{2}\right.\\
& +   & \left.d_{2}\textrm{tr}(\mathbf{c}_{2}\mathbf{c}_{2}^{T}\bar{\mathfrak{I}}_{2}\bar{\mathfrak{I}}_{2}^{T})+\left\Vert \boldsymbol{\bar{\pi}}_{\mathcal{I},3}\right\Vert ^{2}+\left\Vert \boldsymbol{\bar{\pi}}_{\mathcal{R}}\right\Vert ^{2}+d_{0}\textrm{tr}(\mathbf{c}_{0}\mathbf{c}_{0}^{T}\boldsymbol{\bar{\pi}}_{\mathcal{R}}\boldsymbol{\bar{\pi}}_{\mathcal{R}}^{T})\right].\label{eq: final value of the test statistic}
\end{align}

\section{MATHEMATICA commands to generate the $\pi_{k,j,\ell}(r,u)$ of Section~3.1% \ref{subsec:Sec3.1}
}\label{sec:Appeb-B} 

We give the MATHEMATICA commands to generate the $\{\Psi_{k,j}(u)\,|\,j=1,\ldots,e_{m}(k)\}$
of Section \ref{subsec:Sec3.1}. The package \texttt{\textcolor{black}{HFT10.m}}
must first be downloaded from the site given in Appendix~B of \cite{Axler2001} %Axler, Bourdon \& Ramey (2001)
and loaded into MATHEMATICA via the command
\texttt{\textcolor{black}{<\textcompwordmark{}<''.../.../.../HFT10.m''}},
where \texttt{\textcolor{black}{.../.../.../}} is the path leading
to where the package \texttt{\textcolor{black}{HFT10.m}} is stored
on the computer; after typing \texttt{\textcolor{black}{<\textcompwordmark{}<}},
one can use the ``\texttt{\textcolor{black}{File Pat}}h...'' command
in the ``\texttt{\textcolor{black}{Insert}}'' menu to automatically
generate this path.

Once the package is loaded (the text ``{*}\texttt{\textcolor{black}{{}
You can now use the functions in this package.}}'' will appear on
the MATHEMATICA notebook), the user needs only to set the dimension
$m$ via the command \texttt{\textcolor{black}{setDimension{[}u, m{]}}}.
Here $u=(u_{1},...,u_{m})$ is the vector of variables in which the
spherical harmonics will be expressed. Then the command \medskip{}

$\mathtt{\Psi[k\_,u\_]:=\textrm{basisH[}k,u,\textrm{ Sphere}]/.\Vert u\Vert\rightarrow1;}$ 

\medskip{}
generates the basis of dimension $e_{m}(k)$ for $E_{m}(K)$. For
example, typing

\medskip{}

\texttt{\textcolor{black}{setDimension{[}u,5{]};}}

\texttt{\textcolor{black}{Do{[}Print{[}TableForm{[}}}$\mathtt{\Psi(k,u)}$\texttt{\textcolor{black}{{]}{]}},\texttt{\textcolor{black}{{}
\{k,1,4\}{]};}}}

\medskip{}
prints the $\{\Psi_{k,j}(u)\,|\,j=1,\ldots,e_{5}(k)\}$, for $k=1,...,$4.\medskip{}
To generate the $s_{j,i}(r)$ for the case of the null MVN distribution,
the command is

\medskip{}

\texttt{\textcolor{black}{s{[}j\_, i\_, r\_{]} := (-1)\textasciicircum{}j{*}Sqrt{[}j!{*}Gamma{[}m/2{]}/(2\textasciicircum{}i{*}Gamma{[}m/2
+ j + i{]}){]}{*} LaguerreL{[}j, m/2 + i - 1, r\textasciicircum{}2/2{]};}}

\medskip{}
Finally, to generate the $\pi_{k,j,\ell}(r,u)$, the commands are\medskip{}

\texttt{\textcolor{black}{e{[}m\_, k\_{]} := Which{[}k == 0, 1, k
== 1, m, k >= 2, Binomial{[}m + k - 1, m - 1{]} - Binomial{[}m + k
- 3, m - 1{]}{]};}}

$\mathtt{\pi funct}${[}\texttt{\textcolor{black}{k\_, j\_, $\ell$\_,
r\_{]} := r\textasciicircum{}(k - 2{*}j){*}s{[}j, k - 2{*}j, r{]}{*}}}$\mathtt{\Psi[k-2j,u][[\ell]}${]};

CONB$\Pi${[}\texttt{\textcolor{black}{k\_{]} := Flatten{[}Table{[}}}$\mathtt{\pi funct}${[}\texttt{\textcolor{black}{k,
j, $\ell$, r{]}, \{j, 0, Floor{[}k/2{]}\}, \{l, 1, e{[}m, k - 2{*}j{]}\}{]}{]};}}

\medskip{}
 Table \ref{tab:Table des polynomes pour m =00003D 2} lists the polynomials
$\pi_{k,j,\ell}(r,u)$ for the case $m=2,k=3,...,5$ that are used
in the experiment of Section \ref{sec:sec 6 Usefulness of Dx}.

\begin{table}
\centering{}%
\begin{tabular}{|c|c|c|}
\hline 
\textcolor{black}{$\pi_{3,0,1}$ = $\frac{r^{3}u_{1}(3-4u_{1}^{2})}{2\sqrt{6}}$} & \textcolor{black}{$\pi_{3,0,2}$ = $\frac{r^{3}u_{2}(1-4u_{1}^{2})}{2\sqrt{6}}$} & \textcolor{black}{$\pi_{3,1,1}$ = $\frac{r(r^{2}-4)u_{1}}{2\sqrt{2}}$}\tabularnewline
\hline 
\textcolor{black}{$\pi_{3,1,2}$}\textcolor{blue}{{} }\textcolor{black}{=}\textcolor{blue}{{}
}\textcolor{black}{$\frac{r(r^{2}-4)u_{2}}{2\sqrt{2}}$} & \textcolor{black}{$\pi_{4,0,1}$ = $\frac{r^{4}(1-8u_{1}^{2}+8u_{1}^{4})}{8\sqrt{3}}$} & \textcolor{black}{$\pi_{4,0,2}$ = $\frac{r^{4}u_{2}(u_{1}-2u_{1}^{3})}{2\sqrt{3}}$}\tabularnewline
\hline 
\textcolor{black}{$\pi_{4,1,1}$ = $\frac{r^{2}(r^{2}-6)(1-2u_{1}^{2})}{4\sqrt{3}}$} & \textcolor{black}{$\pi_{4,1,2}$ = $\frac{r^{2}(r^{2}-6)u_{1}u_{2})}{2\sqrt{3}}$} & \textcolor{black}{$\pi_{4,2,1}$ = $\frac{r^{4}-8r^{2}+8}{8}$}\tabularnewline
\hline 
\textcolor{black}{$\pi_{5,0,1}$ = $\frac{r^{5}u_{1}(5-20u_{1}^{2}+16u_{1}^{4})}{8\sqrt{30}}$} & \textcolor{black}{$\pi_{5,0,2}$ = $\frac{r^{5}u_{2}(1-12u_{1}^{2}+16u_{1}^{4})}{8\sqrt{30}}$} & \textcolor{black}{$\pi_{5,1,1}$ = $\frac{r^{3}(r^{2}-8)u_{1}(3-4u_{1}^{2})}{8\sqrt{6}}$}\tabularnewline
\hline 
\textcolor{black}{$\pi_{5,1,2}$ = $\frac{r^{3}(r^{2}-8)u_{2}(1-4u_{1}^{2})}{8\sqrt{6}}$} & \textcolor{black}{$\pi_{5,2,1}$ = $\frac{r(r^{4}-12r^{2}+24)u_{1}}{8\sqrt{3}}$} & \textcolor{black}{$\pi_{5,2,2}$ = $\frac{r(r^{4}-12r^{2}+24)u_{2}}{8\sqrt{3}}$}\tabularnewline
\hline 
\end{tabular}\caption{\label{tab:Table des polynomes pour m =00003D 2}Table of the polynomials
$\pi_{k,j,\ell}(r,u)$ of Theorem \ref{thm:3.1} in the case of a
MVN null hypothesis, for $m=2$ and $k=3,$ 4 and 5.}
 
\end{table}

Table \ref{tab:Table des polynomes pour m =00003D 3} lists the $\pi_{k,j,\ell}(r,u)$
required for the Open/Closed book example of Section \ref{sec:sec 6 Usefulness of Dx}
($m=3$ and $k=3,4$ and 5 )

\begin{table}
\centering{}%
\begin{tabular}{|c|c|c|}
\hline 
\textcolor{black}{$\pi_{3,0,1}$ = $\frac{r^{3}u_{1}(1-5u_{2}^{2})}{2\sqrt{10}}$} & \textcolor{black}{$\pi_{3,0,2}$ = $\frac{r^{3}u_{2}(3-5u_{2}^{2})}{2\sqrt{15}}$} & \textcolor{black}{$\pi_{3,0,3}$ = $\frac{r^{3}u_{3}(1-5u_{2}^{2})}{2\sqrt{10}}$}\tabularnewline
\hline 
\textcolor{black}{$\pi_{3,0,4}$}\textcolor{blue}{{} }\textcolor{black}{=}\textcolor{blue}{{}
${\color{black}\frac{r^{3}u_{1}(u_{1}^{2}-3u_{3}^{2})}{2\sqrt{6}}}$} & \textcolor{black}{$\pi_{3,0,5}$ = $\frac{r^{3}u_{2}(u_{1}^{2}-u_{3}^{2})}{2}$} & \textcolor{black}{$\pi_{3,0,6}$ = $\frac{r^{3}u_{3}(3u_{1}^{2}-u_{3}^{2})}{2\sqrt{6}}$}\tabularnewline
\hline 
\textcolor{black}{$\pi_{3,0,7}$ = $r^{3}u_{1}u_{2}u_{3}$} & \textcolor{black}{$\pi_{3,1,1}$ = $\frac{r(r^{2}-5)u_{1}}{\sqrt{10}}$} & \textcolor{black}{$\pi_{3,1,2}$ = $\frac{r(r^{2}-5)u_{2}}{\sqrt{10}}$}\tabularnewline
\hline 
\textcolor{black}{$\pi_{3,1,3}$ = $\frac{r(r^{2}-5)u_{3}}{\sqrt{10}}$} & \textcolor{black}{$\pi_{4,0,1}$ = ${\color{black}\frac{r^{4}(3-30u_{2}^{2}+35u_{2}^{4})}{8\sqrt{105}}}$} & \textcolor{black}{$\pi_{4,0,2}$ = $\frac{r^{4}u_{2}u_{3}(3-7u_{2}^{2})}{2\sqrt{42}}$}\tabularnewline
\hline 
\textcolor{black}{$\pi_{4,0,3}$ = $\frac{r^{4}(7u_{2}^{2}-1)(u_{3}^{2}-u_{1}^{2})}{4\sqrt{21}}$} & \textcolor{black}{$\pi_{4,0,4}$ = $\frac{r^{4}u_{2}u_{3}(3u_{1}^{2}-u_{3}^{2})}{2\sqrt{6}}$} & \textcolor{black}{$\pi_{4,0,5}$ = $\frac{r^{4}(u_{1}^{4}-6u_{1}^{2}u_{3}^{2}+u_{3}^{4})}{8\sqrt{3}}$}\tabularnewline
\hline 
\textcolor{black}{$\pi_{4,0,6}$ = $\frac{r^{4}u_{1}u_{2}(3-7u_{2}^{2})}{2\sqrt{42}}$} & \textcolor{black}{$\pi_{4,0,7}$ = $\frac{r^{4}u_{1}u_{3}(1-7u_{2}^{2})}{4\sqrt{21}}$} & \textcolor{black}{$\pi_{4,0,8}$ = $\frac{r^{4}u_{1}u_{2}(u_{1}^{2}-3u_{3}^{2})}{2\sqrt{6}}$}\tabularnewline
\hline 
\textcolor{black}{$\pi_{4,0,9}$}\textcolor{red}{{} }\textcolor{black}{=
$\frac{r^{4}u_{1}u_{3}(u_{1}^{2}-u_{3}^{2})}{2\sqrt{3}}$} & \textcolor{black}{$\pi_{4,1,1}$}\textcolor{red}{{} }\textcolor{black}{=
$\frac{r^{2}(r^{2}-7)(1-3u_{2}^{2})}{2\sqrt{42}}$} & \textcolor{red}{${\color{black}\pi_{4,1,2}}$ }\textcolor{black}{=
$\frac{r^{2}(r^{2}-7)u_{2}u_{3}}{\sqrt{14}}$}\tabularnewline
\hline 
\textcolor{black}{$\pi_{4,1,3}$ = $\frac{r^{2}(r^{2}-7)(u_{1}^{2}-u_{3}^{2})}{2\sqrt{14}}$} & \textcolor{black}{$\pi_{4,1,4}$ = $\frac{r^{2}(r^{2}-7)u_{1}u_{2}}{\sqrt{14}}$} & \textcolor{black}{$\pi_{4,1,5}$ = $\frac{r^{2}(r^{2}-7)u_{1}u_{3}}{\sqrt{14}}$}\tabularnewline
\hline 
$\pi_{4,2,1}$ = $\frac{(r^{4}-10r^{2}+15)}{2\sqrt{30}}$ & $\pi_{5,0,1}$ =$\frac{r^{5}u_{2}(15-70u_{2}^{2}+63u_{2}^{4})}{24\sqrt{105}}$ & $\pi_{5,0,2}$ =$\frac{r^{5}u_{3}(1-14u_{2}^{2}+21u_{2}^{4})}{24\sqrt{7}}$\tabularnewline
\hline 
$\pi_{5,0,3}$ =$\frac{r^{5}u_{2}(1-3u_{2}^{2})(u_{1}^{2}-u_{3}^{2})}{12}$ & $\pi_{5,0,4}$ =$\frac{r^{5}u_{3}(1-9u_{2}^{2})(3u_{1}^{2}-u_{3}^{2})}{24\sqrt{6}}$ & $\pi_{5,0,5}$ =$\frac{r^{5}u_{2}(9u_{1}^{4}-6u_{1}^{2}u_{3}^{2}-7u_{3}^{4})}{8\sqrt{3}}$\tabularnewline
\hline 
$\pi_{5,0,6}$ =$\frac{r^{5}u_{3}(5u_{1}^{4}-10u_{1}^{2}u_{3}^{2}+u_{3}^{4})}{8\sqrt{30}}$ & $\pi_{5,0,7}$ =$\frac{r^{5}u_{1}(1-14u_{2}^{2}+21u_{2}^{4})}{24\sqrt{7}}$ & $\pi_{5,0,8}$ =$\frac{r^{5}u_{1}u_{2}u_{3}(1-3u_{2}^{2})}{6}$\tabularnewline
\hline 
$\pi_{5,0,9}$ =$\frac{r^{5}u_{1}(1-9u_{2}^{2})(u_{1}^{2}-3u_{3}^{2})}{24\sqrt{6}}$ & $\pi_{5,0,10}$ =$\frac{r^{5}u_{1}u_{2}u_{3}(u_{1}^{2}-u_{3}^{2})}{2\sqrt{3}}$ & $\pi_{5,0,11}$ =$\frac{r^{5}u_{1}(u_{1}^{4}-10u_{1}^{2}u_{3}^{2}+5u_{3}^{4})}{8\sqrt{30}}$\tabularnewline
\hline 
$\pi_{5,1,1}$ =$\frac{r^{3}(r^{2}-9)u_{2}(3-5u_{2}^{2})}{6\sqrt{30}}$ & $\pi_{5,1,2}$ =$\frac{r^{3}(r^{2}-9)u_{3}(1-5u_{2}^{2})}{12\sqrt{5}}$ & $\pi_{5,1,3}$ =$\frac{r^{3}(r^{2}-9)u_{2}(u_{1}^{2}-u_{3}^{2})}{6\sqrt{2}}$\tabularnewline
\hline 
$\pi_{5,1,4}$ =$\frac{r^{3}(r^{2}-9)u_{3}(3u_{1}^{2}-u_{3}^{2})}{12\sqrt{3}}$ & $\pi_{5,1,5}$ =$\frac{r^{3}(r^{2}-9)u_{1}(1-5u_{2}^{2})}{12\sqrt{5}}$ & $\pi_{5,1,6}$ =$\frac{r^{3}(r^{2}-9)u_{1}u_{2}u_{3}}{3\sqrt{2}}$\tabularnewline
\hline 
$\pi_{5,1,7}$ =$\frac{r^{3}(r^{2}-9)u_{1}(u_{1}^{2}-3u_{3}^{2})}{12\sqrt{3}}$ & $\pi_{5,2,1}$ =$\frac{r(r^{4}-14r^{2}+35)u_{1}}{2\sqrt{70}}$ & $\pi_{5,2,2}$ =$\frac{r(r^{4}-14r^{2}+35)u_{2}}{2\sqrt{70}}$\tabularnewline
\hline 
$\pi_{5,2,3}$ =$\frac{r(r^{4}-14r^{2}+35)u_{3}}{2\sqrt{70}}$ & \multicolumn{1}{c}{%
} & \multicolumn{1}{c}{%
}\tabularnewline
\cline{1-1} 
\end{tabular}\caption{\label{tab:Table des polynomes pour m =00003D 3}Table of the polynomials
$\pi_{k,j,\ell}(r,u)$ of Theorem \ref{thm:3.1} in the case of a
MVN null hypothesis, for $m=3$ and $k=3,$ 4 and 5.}
\end{table}

\section{Smooth test for the bivariate logistic and Pearson type II distributions} % \ref{subsec:Sec3.1}
\label{sec:Appeb-C} 

\subsection{The bivariate logistic distribution}\label{subsec:Section bivariate Logistic}

We consider the bivariate logistic distribution in \cite{Lemonte2011} with density generator $\phi_{m}(y)=e^{-y}/(1+e^{-y})^{2}$.
This is again a competitor to the MVN but with shorter tails. The
moments of $R^{2}$ have explicit but complicated expressions and
it is shorter to use numerical approximations. Consequently, the $s_{j,k-2j}(\cdot)$
required for $\mathcal{Q}_{5}$ are : $s_{0,1}(r)=0.849322$; $s_{0,2}(r)=0.551329$;
$s_{0,3}(r)=0.30403$; $s_{0,4}(r)=0.148319$; $s_{0,5}(r)=0.0654688$;
$s_{1,0}(r)=-1.18523+0.854964\,r^{2}$; $s_{1,1}(r)=-1.36758+0.576276\,r^{2}$;
$s_{1,2}(r)=-1.0461+0.318116\,r^{2}$; $s_{1,3}(r)=-0.646026+0.153749\,r^{2}$;
$s_{2,0}(r)=1.24468-1.86588\,r^{2}+0.407913\,r^{4}$ and $s_{2,1}(r)=1.79127-1.51119\,r^{2}+0.230011\,r^{4}$.
Next, $\zeta(r^{2})=2r^{2}\tanh(r^{2}/2)$, $\sigma_{1}=3.18173$,
$\sigma_{2}=0.82306$ and
\begin{quote}
\begin{alignat*}{1}
\sqrt{d_{0}}\mathbf{c}_{0} & =\{-9.45511,-0.77618\},\\
\sqrt{d_{1}}\mathbf{c}_{1} & =\{12.55,3.40145,-1.89893\},\\
\sqrt{d_{2}}\mathbf{c}_{2} & =\{-8.04922,-1.2599\}.
\end{alignat*}
\end{quote}
Continue as in the bivariate Laplace case.

\subsection{The bivariate Pearson type II distribution}\label{subsec:Section bivariate P=0000EBarson Type 2}

We consider the bivariate Pearson type II distribution described in
\cite[Sec~6.2]{Johnson1987} with density generator $\phi_{m}(y)=(1-y)^{\alpha}$
with $y\in[0,1]$. For $\alpha>0,$ this is another EC distribution
that somewhat resembles the MVN. The $j-$th moment of $R^{2}$ is
$(2j+\alpha)B(2+\alpha,1+j)$ and from (\ref{eq:3.5}), the $s_{j,k-2j}(\cdot)$
required for $\mathcal{Q}_{5}$ are : $s_{0,j}(r)=(\sqrt{(j+1)B(2+\alpha,j+1})^{-1}$
for $j=1,...,5$; $s_{1,j}(r)=\frac{(2+j+\alpha)r^{2}-(1+j)}{\sqrt{B(2+\alpha,j+2)(1+\alpha)(2+j+\alpha)}}$
for $j=0,...,3$; $s_{2,0}(r)=$ $\frac{(2+r^{2}(3+\alpha)(r^{2}(4+\alpha)-4)\sqrt{5+\alpha}}{2\sqrt{1+\alpha}}$
and $s_{2,1}(r)=\frac{(6+r^{2}(4+\alpha)(r^{2}(5+\alpha)-6)\sqrt{(3+\alpha)(6+\alpha)}}{2\sqrt{3}\sqrt{1+\alpha}}$.
Next, because $\zeta(r^{2})=\frac{2\alpha r^{2}}{1-r^{2}}$, one finds
$\sigma_{1}=\frac{4\alpha(\alpha+1)}{\alpha-1}$ , $\sigma_{2}=1-\alpha^{-1}$
(hence $\alpha>1$ in the sequel) and 
\begin{quote}
\begin{alignat*}{1}
\sqrt{d_{0}}\mathbf{c}_{0} & =-\frac{\sqrt{\alpha-1}}{3}\{\frac{(\alpha+2)}{2}\sqrt{(\alpha+3)},\sqrt{\alpha+5}\},\\
\sqrt{d_{1}}\mathbf{c}_{1} & =\frac{\sqrt{\alpha-1}}{2}\{\frac{\sqrt[3]{\alpha+2}\sqrt{(\alpha+1)(\alpha+3)}}{6},\frac{\sqrt{(\alpha+2)(\alpha+3)(\alpha+4)}}{3\sqrt{2}},\frac{\sqrt{\alpha+6}}{\sqrt{3}}\},\\
\sqrt{d_{2}}\mathbf{c}_{2} & =-\frac{\sqrt{\alpha-1}}{2\sqrt{2}}\{\frac{(3+\alpha)\sqrt{\alpha+1}}{\sqrt{3}},\sqrt{5+\alpha}\}.
\end{alignat*}
\end{quote}
Continue as in the previous cases. Note that here $\alpha$ is supposed
known, but the above can serve to prefigure the difficulties to be
encountered with an unknown shape parameter.

\section{An application} % \ref{subsec:Sec3.1}
\label{sec:Appeb-D} 

As an application of the methods of the paper, consider the \textquotedblleft Open-book
closed-book examination\textquotedblright{} data set \cite[p.~3-4]{Mardia1979} corresponding to examination marks in Mechanics,
Vectors, Algebra, Analysis and Statistic for a sample of 88 students.
Here we consider only the marks in Vectors, Algebra and Statistics,
so that \emph{m} = 3 and we wish to test a trivariate MVN. We apply
our test strategy with\emph{ }$K=5$. All polynomials required for
these computations appear in Table \ref{tab:Table des polynomes pour m =00003D 3}.
The results are shown in Table \ref{tab:Table 1} which lists test
statistic $\mathcal{Q}_{5}$ and its scaled $\mathcal{UIR}$ components,
as well as the \emph{p}\textendash values obtained from the reference
asymptotic $\chi^{2}$ distribution and from Monte Carlo (20 000 replications)
approximations. The null hypothesis of trivariate MVN is rejected
at the 5\% level. Inspection of the scaled components shows that both
the distributions of $R$ and $\boldsymbol{U}$ appear consonant with
the MVN and that rejection comes from correlations between \emph{R}
and $\boldsymbol{U}$. If the confidence build in this section for
$m=2$ can be transferred to the case $m=3$, a refined model should
try to take into account the dependencies between these random quantities. 

\begin{table}
\begin{centering}
\begin{tabular}{|c|c|c|c|c|}
\hline 
Test Statistics
 & 
Value
 & \emph{d.f}. & \emph{p}-value $\chi^{2}$ & \emph{p}-value MC\tabularnewline
\hline 
\hline 
$\mathcal{Q}_{5}$ & 98.62 & 46 & 0.004 & %
0.027%
\tabularnewline
\hline 
$\mathcal{U}_{5}^{(s)}$ & %
37.38%
 & 27 & %
0.088%
 & %
0.163%
\tabularnewline
\hline 
$\mathcal{I}_{5}^{(s)}$ & %
33.14%
 & 18 & %
0.016%
 & %
0.028%
\tabularnewline
\hline 
$\mathcal{R}_{5}^{(s)}$ & %
0.77%
 & 1 & %
0.381%
 & %
0.466%
\tabularnewline
\hline 
\end{tabular}
\par\end{centering}
\caption{\label{tab:Table 1}Examination marks ($n=88$) in Vectors, Algebra
and Statistics from the ``Open book-Closed book examination'' data
set \cite[p.~3-4]{Mardia1979}. The null hypothesis is
a trivariate normal distribution; \emph{d.f. }refers to degrees of
freedom of the asymptotic $\chi^{2}$ approximation, \emph{p-value
}$\chi^{2}$ are computed from these reference distributions, while
\emph{p}-value MC refers to the p-value computed from a Monte Carlo
approximation (20 000 replications)}
\end{table}

\end{document}